\documentclass[11pt]{article}
\usepackage[left=1in,top=1in,right=1in,bottom=1in,letterpaper]{geometry}

\usepackage[ruled,vlined]{algorithm2e}
\usepackage{algorithmic}
\usepackage{multirow}
\usepackage{lipsum}
\usepackage{graphicx,amsmath,amsthm}
\graphicspath{{Plots/}}
\usepackage{epstopdf}
\usepackage{hyperref}

\ifpdf
  \DeclareGraphicsExtensions{.eps,.pdf,.png,.jpg}
\else
  \DeclareGraphicsExtensions{.eps}
\fi
\usepackage{microtype}
\usepackage{subfigure}
\usepackage{booktabs} % for professional tables
\usepackage{amsfonts}
\usepackage{makecell}
\usepackage{comment}
\usepackage{enumitem}
\usepackage{tikz}
\usepackage{url}
\usepackage{authblk}
\usepackage[capitalise]{cleveref}
\usepackage{backref}
\renewcommand*{\backref}[1]{\ifx#1\relax \else Page #1 \fi}
\renewcommand*{\backrefalt}[4]{%
  \ifcase #1 \footnotesize{(Not cited.)}%
  \or        \footnotesize{(Cited on page~#2.)}%
  \else      \footnotesize{(Cited on pages~#2.)}%
  \fi
}

\Crefname{ALC@unique}{Line}{Lines}

\newcommand{\M}{\mathcal{M}}

\newcommand{\br}{\mathbb{R}}

\newcommand{\St}{\mathrm{St}}
\newcommand{\grad}{\mathrm{grad}}

\newcommand{\prox}{\mathrm{prox}}
\newcommand{\proj}{\operatorname{Proj}}
\newcommand{\argmin}{\operatorname{argmin}}
\newcommand{\Exp}{\operatorname{Exp}}

\newcommand{\X}{\mathcal{X}}

\newcommand{\Tr}{\mathrm{Tr}}
\newcommand{\st}{\mathrm{ s.t. }}
\newcommand{\retr}{\mathrm{Retr}}
\newcommand{\Proj}{\mathrm{Proj}}

\newcommand{\R}{\mathbb{R}}

\newcommand{\LLg}{\mathcal{L}_{\rho,\gamma}}

\newcommand{\etal}{et al. }
\newcommand{\by}{\mathbf{y}}
\newcommand{\be}{\begin{equation}}
\newcommand{\ee}{\end{equation}}
\newcommand{\bx}{\mathbf{x}}
\newcommand{\CH}{\mathcal{H}}
\newcommand{\CO}{\mathcal{O}}
\newcommand{\CS}{\mathcal{S}}
\newcommand{\T}{\mathrm{T}}

\allowdisplaybreaks
\newtheorem{theorem}{Theorem}%[section]
\newtheorem{lemma}{Lemma}%[section]
\newtheorem{definition}{Definition}%[section]
\newtheorem{assumption}{Assumption}%[section]
\crefname{hypothesis}{Hypothesis}{Hypotheses}

\title{A Riemannian ADMM}

\author[1]{Jiaxiang Li}
\author[2]{Shiqian Ma}
\author[3]{Tejes Srivastava}
\affil[1]{Department of Electrical and Computer Engineering, University of Minnesota}
\affil[2]{Computational Applied Mathematics and Operations Research, Rice University}
\affil[3]{Department of Computer Science, University of Chicago}

\date{\today}

\begin{document}

\maketitle

\begin{abstract}
  We consider a class of Riemannian optimization problems where the objective is the sum of a smooth function and a nonsmooth function, considered in the ambient space. This class of problems finds important applications in machine learning and statistics such as the sparse principal component analysis, sparse spectral clustering, and orthogonal dictionary learning. We propose a Riemannian alternating direction method of multipliers (ADMM) to solve this class of problems. Our algorithm adopts easily computable steps in each iteration. The iteration complexity of the proposed algorithm for obtaining an $\epsilon$-stationary point is analyzed under mild assumptions. Existing ADMM for solving nonconvex problems either does not allow nonconvex constraint set, or does not allow nonsmooth objective function. Our algorithm is the first ADMM type algorithm that minimizes a nonsmooth objective over manifold -- a particular nonconvex set. Numerical experiments are conducted to demonstrate the advantage of the proposed method.
\end{abstract}

\section{Introduction}\label{sec_intro}
Optimization over Riemannian manifolds has drawn a lot of attention due to its applications in machine learning and related disciplines, including low-rank matrix completion \cite{boumal2011rtrmc,vandereycken2013low}, phase retrieval \cite{bendory2017non,sun2018geometric}, blind deconvolution \cite{huang2018blind} and dictionary learning \cite{cherian2016riemannian,Sun-CDR-part1-2017}. Riemannian optimization aims at minimizing an objective function over a Riemannian manifold. When the objective function is smooth, people have proposed to solve them using Riemannian gradient method, Riemannian quasi-Newton method, Riemannian trust-region method, etc. Work along this line has been summarized in the monographs \cite{absil2009optimization,boumal2020introduction} as well as some other references. Recently, due to increasing demand from application areas such as machine learning, statistics, signal processing and so on, there is a line of work designing efficient and scalable algorithms for solving Riemannian optimization problems with nonsmooth objectives. For example, people have studied Riemannian subgradient method \cite{li2019weakly}, Riemannian proximal gradient method \cite{chen2020proximal,Huang-Wei-RPG}, Riemannian proximal point algorithm \cite{chen2019manifold}, Riemannian proximal-linear algorithm \cite{wang2020manifold}, zeroth-order Riemannian algorithms \cite{Li-ZO-MOR}, and so on.

One thing that has not been widely considered is how to design alternating direction method of multipliers (ADMM) on manifolds. ADMM can be a perfect solver for the following nonsmooth optimization over Riemannian manifolds:
\begin{equation}\label{problem_nonsmooth}
\begin{split}
        \min_{x} & \ F(x) := f(x)+g(A x) \\
        \st & \ x\in\M,
\end{split}
\end{equation}
where $f$ is smooth and possibly nonconvex, $g$ is nonsmooth but convex, $\M$ is an embedded submanifold in $\br^n$, and matrix $A\in\br^{m\times n}$. Throughout this paper, the smoothness, Lipschitz continuity, and convexity of functions are interpreted as the functions are being considered in the ambient Euclidean space. If $\M=\br^n$, then problem
\eqref{problem_nonsmooth} reduces to the Euclidean case, and there exist efficient methods such as proximal gradient method, accelerated proximal gradient method, and ADMM for solving it. If the nonsmooth function vanishes, i.e., $g\equiv 0$, then problem \eqref{problem_nonsmooth} reduces to a smooth problem over manifold, and it can be solved by various methods for smooth Riemannian optimization. Therefore, the main challenge of solving \eqref{problem_nonsmooth} lies in the fact that there exist both manifold constraint and nonsmooth objective in the problem. As a result, a very natural idea to deal with this situation is to split the difficulty caused by the manifold constraint and nonsmooth objective. In particular, one can introduce an auxiliary variable $y$ and rewrite \eqref{problem_nonsmooth} equivalently as
\begin{equation}\label{problem_nonsmooth_splitting}
\begin{split}
        \min_{x,y} & \ f(x)+g(y) \\
        \st & \ A x=y,\ x\in\M.
\end{split}
\end{equation}
ADMM is a good candidate for solving \eqref{problem_nonsmooth_splitting}, because it can deal with the nonsmooth objective and the manifold constraint separately and alternately. Here we point out that there exist many ADMM-like algorithms for problems with nonconvex objective \cite{Jiang-Lin-Ma-Zhang-nonconvex-2016,yang2017alternating,themelis2020douglas, zhang2022iteration}. However, these algorithms do not allow the constraint set to be nonconvex. Therefore, they do not apply to the case where manifold constraints are present. In the following, we give a brief literature review on ADMM-like algorithms that allow manifold constraint -- a nonconvex constraint set.

The idea of splitting the nonsmooth objective and manifold constraint in \eqref{problem_nonsmooth} is not new. The first algorithm for this purpose is the SOC (splitting orthogonality constraints) algorithm proposed by Lai and Osher \cite{lai2014splitting}. SOC for solving \eqref{problem_nonsmooth} splits the problem in the following way:
\begin{equation}\label{problem_nonsmooth_soc_splitting}
\begin{split}
        \min_{x,y} & \ f(x)+g(A x) \\
        \st & \ x=y,\ y\in\M,
\end{split}
\end{equation}
and iterates as follows:
\begin{equation}\label{soc}
\begin{split}
        x^{k+1} := &  \ \argmin_x f(x)+g(A x)+\langle \lambda^k, x-y^k\rangle +\frac{\rho}{2}\|x - y^{k}\|_{2}^2 \\
        y^{k+1} := & \ \argmin_{y\in\M} \langle\lambda^k,x^{k+1}-y\rangle+\frac{\rho}{2}\|x^{k+1}-y\|_{2}^2 \\
        \lambda^{k+1} := & \ \lambda^{k} + \rho(x^{k+1} - y^{k+1}),
\end{split}
\end{equation}
where $\lambda$ denotes the Lagrange multiplier and $\rho>0$ is a penalty parameter. Note that the $x$-subproblem in \eqref{soc} is an unconstrained problem, which can be solved by proximal gradient method and many others, and the $y$-subproblem corresponds to a projection onto the manifold $\M$.
A closely related algorithm named MADMM (manifold ADMM), proposed in \cite{kovnatsky-madmm-2016} for solving \eqref{problem_nonsmooth_splitting}, iterates as follows:
\begin{equation}\label{madmm}
\begin{split}
        x^{k+1} := &  \ \argmin_{x\in\M} f(x)+\langle \lambda^k,Ax-y^k\rangle+\frac{\rho}{2}\|A x - y^{k}\|_{2}^2 \\
        y^{k+1} := & \ \argmin_{y} g(y) +\langle\lambda^k,Ax^{k+1}-y \rangle+ \frac{\rho}{2}\|A x^{k+1} - y\|_{2}^2 \\
        \lambda^{k+1} := & \ \lambda^{k} + \rho(A x^{k+1} - y^{k+1}).
\end{split}
\end{equation}
In \eqref{madmm}, the $x$-subproblem is a Riemannian optimization with smooth objective which can be solved by Riemannian gradient method, and the $y$-subproblem corresponds to the proximal mapping of function $g$. However, there lacks convergence guarantees for both SOC and MADMM.

When the nonsmooth term in \eqref{problem_nonsmooth} vanishes, i.e., $g\equiv 0$, an ADMM for nonconvex optimization can be used to solve \eqref{problem_nonsmooth} as illustrated in \cite{wang2019global}. Since $g\equiv 0$, the problem \eqref{problem_nonsmooth} reduces to
\begin{equation}\label{problem-smooth-manifold-split}
\begin{split}
        &\min_{x,y}\ f(x) + I_{\M}(y) \\
        &\st, \ x=y,
\end{split}
\end{equation}
where $I_{\M}$ is the indicator function of manifold $\M$. The ADMM for solving \eqref{problem-smooth-manifold-split} iterates as follows:
\begin{equation}\label{admm_nonconvex}
\begin{split}
        x^{k+1} := & \ \argmin_x f(x)+\langle \lambda^{k},x-y^k\rangle+\frac{\rho}{2}\|x - y^{k} \|_{2}^2 \\
        y^{k+1} := & \ \argmin_{y\in\M}\langle \lambda^{k},x^{k+1}-y\rangle+ \frac{\rho}{2} \| x^{k+1}-y\|_{2}^2 \\
        \lambda^{k+1} := & \ \lambda^{k} + \rho(x^{k+1} - y^{k+1}).
\end{split}
\end{equation}
The convergence of \eqref{admm_nonconvex} is established in \cite{wang2019global} under the assumption that $f$ is Lipschitz differentiable. Note that the convergence only applies when $g\equiv 0$. The ADMM studied in \cite{wang2019global} does not apply to \eqref{problem_nonsmooth} when the nonsmooth function $g$ presents.

Another ADMM was proposed in \cite{lu2018nonconvex} for solving a particular smooth Riemannian optimization problem: the sparse spectral clustering. This problem can be cast below.
\begin{equation}\label{ssc}
\begin{split}
    \min_{P,U} & \ \langle L,U U^\top \rangle + g(P), \\
    \st & \ P=U U^\top,U^\top U=I,
\end{split}
\end{equation}
where $L$ is a given matrix, $g$ is a smooth function that promotes the sparsity of $UU^\top$. The ADMM for solving \eqref{ssc} iterates as follows.
\begin{equation}\label{ssc_admm}
\begin{split}
        U^{k+1} := & \ \argmin_{U^\top U=I} \langle L,UU^\top\rangle + \langle\Lambda^k,P^k-UU^\top\rangle + \frac{\rho}{2}\|P^k-U U^\top\|_F^2 \\
        P^{k+1} := & \ \argmin_P g(P) + \langle\Lambda^k,P-U^{k+1}(U^{k+1})^\top\rangle+\frac{\rho}{2}\|P - U^{k+1} (U^{k+1})^\top \|_F^2\\
        \Lambda^{k+1} := & \ \Lambda^{k} + \rho(P^{k+1} - U^{k+1} (U^{k+1})^\top). %\\
        %\rho_{k+1} := & \min\{\rho_{\max}, \delta\rho_k\},
\end{split}
\end{equation}
Note that the ADMM in \cite{lu2018nonconvex} requires the smoothness on the objective function as well, and it does not apply to the case where the objective function is nonsmooth. Zhang \etal \cite{zhang2019primal} proposed a proximal ADMM which solves the following problem:
\begin{equation}\label{problem-zhang-ma-zhang}
\begin{split}
    \min & \ f(x_1,\ldots,x_N) + \sum_{i=1}^{N-1} g_i(x_i) \\
    \st  & \ x_N = b - \sum_{i=1}^{N-1} A_i x_i \\
         & \ x_i \in \M_i \cap \mathcal{X}_i, i=1,\ldots,N-1,
\end{split}
\end{equation}
where $f$ is a smooth function, $g_i$ is a nonsmooth function, $\M_i$ is a Riemannian manifold, and $\mathcal{X}_i$ is a convex set. The authors of \cite{zhang2019primal} established the iteration complexity of the proposed proximal ADMM for obtaining an $\epsilon$-stationary point of \eqref{problem-zhang-ma-zhang}. A notable requirement in \eqref{problem-zhang-ma-zhang} is that the last block variable (i.e., $x_N$) must not appear in the nonsmooth part of the objective, nor be subject to manifold constraints. This is in sharp contrast to problem \eqref{problem_nonsmooth_splitting}, where one block variable is associated with the manifold constraint, and the other block variable is associate with the nonsmooth part of the objective.

Other than ADMM-type algorithms, there also exist some other algorithms for solving \eqref{problem_nonsmooth}. Here we briefly discuss two of them: Riemannian subgradient method and Riemannian proximal gradient method. Because the objective function of \eqref{problem_nonsmooth} is nonsmooth, it is a natural idea to use Riemannian subgradient method \cite{ferreira1998subgradient,borckmans2014riemannian,grohs2016varepsilon,hosseini2015convergence,hosseini2017riemannian,hosseini2018line,grohs2016nonsmooth,li2019weakly}
to solve it. The Riemannian subgradient method for solving \eqref{problem_nonsmooth} updates the iterate by
\[
x^{k+1}=\retr_{x^{k}}(- \eta_k v^k),
\]
where $v^k$ is a Riemannian subgradient of $F$ at $\M$, $\eta_k>0$ is a stepsize, and $\retr$ denotes the retraction operation. Convergence of this method is established in \cite{ferreira1998subgradient} when $F$ is geodesically convex, and iteration complexity is analyzed in \cite{li2019weakly} when $F$ is weakly convex over the Stiefel manifold. Another representative algorithm for solving \eqref{problem_nonsmooth} is the manifold proximal gradient method (ManPG), which was proposed recently by Chen \etal \cite{chen2020proximal}. A typical iteration of ManPG is given below:
\begin{equation}\label{manpg}
\begin{split}
        v^k  & :=\argmin_{v\in \T_{x^{k}}\M} \langle\grad f(x^k), v\rangle + \frac{1}{2 t}\|v\|^2+g(A(x^k + v))\\
        x^{k+1} & := \retr_{x^k}(\alpha v^k),
\end{split}
\end{equation}
where $t>0$ and $\alpha>0$ are stepsizes, $\T_x\M$ denotes the tangent space of $\M$ at $x$, and $\grad f$ denotes the Riemannian gradient of $f$. Chen \etal \cite{chen2020proximal} analyzed the iteration complexity of ManPG for obtaining an $\epsilon$-stationary point of \eqref{problem_nonsmooth}. Moreover, Chen \etal \cite{chen2020proximal} suggested to solve the subproblem for determining $v_k$ in \eqref{manpg} by a semi-smooth Newton method \cite{chen2020proximal,xiao2018regularized}. Huang and Wei extended ManPG to more general manifold \cite{Huang-Wei-RPG} and they also designed an accelerated ManPG \cite{huang2019extension} that demonstrates superior numerical behaviour than the original ManPG. In a more recent work, Zhou \etal \cite{zhou2022semismooth} proposed an augmented Lagrangian method that solves the manifold constrained problems with nonsmooth objective. Note that similar to ManPG, the algorithms in \cite{Huang-Wei-RPG,huang2019extension,zhou2022semismooth} all require solving a relatively difficult subproblem which needs to be solved by semi-smooth Newton algorithm. In this paper we do not need to deal with difficult subproblems -- all steps of our algorithms are explicit and easy-to-compute. 

In Table \ref{table_summary}, we give a detailed comparison of our algorithm with some existing methods. From Table \ref{table_summary} we see that our method allows nonsmooth objective, simple update, and has convergence rate analysis. This combination has been missing in the literature. More specifically, ADMM \cite{wang2019global} was only analyzed for smooth objective. SOC \cite{lai2014splitting} and MADMM \cite{kovnatsky-madmm-2016} lack convergence analysis. ADMM-NSSC \cite{lu2018nonconvex} was designed for a specific application, and the convergence rate was only given for the difference of consecutive iterates, which is not enough to measure the optimality. prox-ADMM \cite{zhang2019primal} requires a strong assumption on the problem: the last block variable cannot appear in the nonsmooth objective. In terms of the per-iteration complexity, it is clear that RADMM admits simple updates: retraction for the manifold and proximal mapping for the nonsmooth function. SOC, MADMM and prox-ADMM all have higher per-iteration complexity than RADMM. 
%Moreover, SOC, MADMM and prox-ADMM need to solve a smooth manifold optimization in each iteration, which might be time-consuming in certain applications. Our RADMM overcomes all these drawbacks. 
Although the iteration complexity of RADMM is seemingly higher than some other algorithms, we point out that different papers use different notions of $\epsilon$-stationarity, and these notions may not be equivalent to each other. Therefore, the iteration complexity results in the last column of Table \ref{table_summary} are not comparable.

\begin{table}[ht]
%\vskip 0.15in
\centering
\begin{tabular}{c || p{0.6in} | p{2in} | c } 
 \hline
 Algorithm & Problem & Per iteration & Convergence (rate) \\
 \hline\hline 
 ADMM \cite{wang2019global} & \eqref{problem_nonsmooth} (with $g=0$) & projection onto manifold, proximal mapping of smooth function  & $\mathcal{O}(1/\epsilon^2)$ \\ 
 \hline
 SOC \cite{lai2014splitting} & \eqref{problem_nonsmooth_soc_splitting} & projection onto manifold, difficult Euclidean subproblem & NA \\
 \hline
 MADMM \cite{kovnatsky-madmm-2016} & \eqref{problem_nonsmooth_splitting} & difficult manifold subproblem, proximal mapping of nonsmooth function & NA \\
 \hline
 ADMM-NSSC \cite{lu2018nonconvex} & \eqref{ssc} & projection onto manifold, proximal mapping of nonsmooth function & Asymptotic* \\
 \hline
 prox-ADMM \cite{zhang2019primal} & \eqref{problem-zhang-ma-zhang} & difficult manifold subproblem & $\mathcal{O}(1/\epsilon^2)$ \\
 \hline
 RADMM (this work) & \eqref{problem_nonsmooth_splitting} & retraction, proximal mapping of nonsmooth function & $\mathcal{O}(1/\epsilon^4)$ \\
 \hline
\end{tabular}
\vskip 0.02in
\footnotesize{*Require strong assumption on the sequence to achieve rate of convergence.}
\vskip 0.1in
\caption{Summary of the convergence results for different ADMM-based algorithms to solve \eqref{problem_nonsmooth}, \eqref{problem_nonsmooth_splitting} or \eqref{problem_nonsmooth_soc_splitting} (Note that these three problems are equivalent).}
\label{table_summary}
%\vskip -0.1in
\end{table}

\textbf{Our contributions.} In this paper, we propose a Riemannian ADMM (RADMM) for solving \eqref{problem_nonsmooth_splitting} based on a Moreau envelope smoothing technique. Our RADMM for solving \eqref{problem_nonsmooth_splitting} contains easily computable steps in each iteration. We analyze the iteration complexity of our RADMM for obtaining an $\epsilon$-stationary point to \eqref{problem_nonsmooth_splitting} under mild assumptions. Existing ADMM for solving nonconvex problems either does not allow nonconvex constraint set, or does not allow nonsmooth objective function. In contrast, our algorithm (and its convergence result) is established for problems with simultaneous nonsmooth objective and manifold constraint. Numerical results of the proposed algorithm for solving sparse principal component analysis and dual principal component pursuit are reported, which demonstrate its superiority over existing methods.

\textbf{Organizations.} The rest of this paper is organized as follows. We propose our RADMM in Section \ref{sec_algo}, whose iteration complexity is analyzed in Section \ref{sec_convergence}. Section \ref{sec_numeric} is devoted to applications and numerical experiments. We draw some concluding remarks in Section \ref{sec_conclusion}.

\textbf{Notation.} We use $\|\cdot\|$ to denote the common Euclidean norm for vectors. For matrices, we use $\|\cdot\|_2$ to denote the operator-2 norm, i.e. $\|A\|_2:=\max_{\|x\|=1} \|A x\|$. Note that we use $\|\cdot\|_1$ to denote the $\ell_1$ matrix norm, i.e. $\|X\|_1=\sum_{ij}|X_{ij}|$. $\Tr(X)$ denotes the trace of matrix $X$.

\section{A Riemannian ADMM} \label{sec_algo}

In this section, we introduce our Riemannian ADMM algorithm. We first review some basics of Riemannian optimization.

\subsection{Basics on Riemannian optimization}

Let $\M\subset\R^n$ be a differentiable embedded submanifold. We have the following definition for the tangent space.
\begin{definition}[Tangent space]
    Consider a manifold $\M$ embedded in a Euclidean space. For any $x\in\M$, the tangent space $\T_x\M$ at $x$ is a linear subspace that consists of the derivatives of all differentiable curves on $\M$ passing through $x$:
    \begin{equation}\label{eq_tangent_space}
%    \begin{split}
        \T_x\M=\{\gamma^{\prime}(0): \gamma(0)=x, \gamma([-\delta, \delta]) \subset \mathcal{M}\text { for some } \delta>0, \gamma \text { is differentiable}\}.
%    \end{split}
    \end{equation}
%    The tangent space always has the same dimension as $\M$.
\end{definition}

The manifold $\M$ is a Riemannian manifold if it is equipped with an inner product on the tangent space, $\langle \cdot, \cdot \rangle _x : \T_x\M \times \T_x\M \rightarrow \R$, that varies smoothly on $\M$. As an example, consider the Stiefel manifold $\M= \St(n, p):=\{X\in\R^{n\times p}: X^\top X=I_p\}$. The tangent space of $\St(n, p)$ is given by $\T_X\M=\{Y\in\R^{n\times p}: X^\top Y+Y^\top X=0\}$. It is easy to verify that the projection onto the tangent space of $\St(n, p)$ is $\proj_{\T_X\M}(Y) = (I-X X^\top)Y + X\operatorname{skew}(X^\top Y)$, where $\operatorname{skew}(A):=(A-A^\top)/2$. We refer to the monographs \cite{absil2009optimization,boumal2020introduction} for more examples. We now introduce the concept of a Riemannian gradient.
\begin{definition}[Riemannian Gradient]\label{def_riemann_grad}
    Suppose $f$ is a smooth function on $\M$. The Riemannian gradient $\grad f(x)$ is a vector in $\T_x\M$ satisfying $\left.\frac{d(f(\gamma(t)))}{d t}\right|_{t=0}=\langle v, \grad f(x)\rangle_{x}$ for any $v\in \T_x\M$, where $\gamma(t)$ is a curve as described in \eqref{eq_tangent_space}.
\end{definition}

Another useful concept is the retraction.
\begin{definition}[Retraction]\label{def_retraction}
    A retraction mapping $\retr_x$ is a smooth mapping from $\T_x\M$ to $\M$ (not necessary injective or surjective) such that: $\retr_x(0)=x$, where $0$ is the zero element of $\T_x\M$, and the differential of $\retr_x$ at $0$ is an identity mapping, i.e., $\left.\frac{d \retr_x(t\eta)}{d t}\right|_{t=0}=\eta$, $\forall \eta\in \T_x\M$. In particular, the exponential mapping $\Exp_x$ is a retraction that generates geodesics.
\end{definition}
In the theoretical analysis of our algorithm, we always assume that the retraction is injective from $\T_x\M$ to $\M$ for any point $x\in\M$, thus the existence of the inverse of the retraction function $\retr_{x}^{-1}$ is guaranteed. For example, when $\M$ is complete, the exponential mapping $\Exp_x$ (which is a special example of retraction) is always defined for every $\xi\in \T_x\M$, and the inverse of the exponential mapping $\Exp_{x}^{-1}(y)\in \T_x\M$ (which is called the logarithm mapping), is always well defined for any $x,y\in\M$.

Throughout this paper, we consider the Riemannian metric on $\M$ that is induced from the Euclidean inner product; i.e., for any $\xi,\eta \in \T_x\M$, we have $\langle\xi,\eta\rangle_x = \Tr(\xi^\top\eta)$.
The Euclidean gradient of a smooth function $f$ is denoted as $\nabla f$ and the Riemannian gradient of $f$ is denoted as $\grad\,f$. Note that by our choice of the Riemannian metric, we have $\grad\,f(x)=\Proj_{\T_x\M}\nabla f(x)$, the orthogonal projection of $\nabla f(x)$ onto the tangent space.

\subsection{Our Riemannian ADMM}
Now we are ready to introduce our RADMM algorithm. Our RADMM for solving \eqref{problem_nonsmooth_splitting} is based on the Moreau envelope smoothing technique. In particular, we consider smoothing the function $g$ in \eqref{problem_nonsmooth_splitting} by adding a quadratic proximal term, which leads to:

%We thus consider the smoothed version problem:   \red{cite Jinshan Zeng, Wotao Yin, Ding-xuan Zhou}
\begin{equation}\label{Moreau_envelope-3block}
\begin{split}
        \min_{x,y,z} & \ f(x)+ g(y) + \frac{1}{2\gamma}\|y-z\|^2 \\
        \st & \ A x=z,\ x\in\M,
\end{split}
\end{equation}
where $\gamma>0$ is a parameter. Equivalently, \eqref{Moreau_envelope-3block} can also be rewritten as
\begin{equation}\label{Moreau_envelope}
\begin{split}
        \min_{x,z} & \ f(x)+g_{\gamma}(z) \\
        \st & \ A x=z,\ x\in\M,
\end{split}
\end{equation}
where $g_{\gamma}(z)=\min_{y}\left\{ g(y) + \frac{1}{2\gamma}\|y-z\|^2 \right\}$ is the Moreau envelope of $g$, and it is known that $g_\gamma$ is a smooth function when $g$ is convex.

We need to point out that the idea of Moreau envelope smoothing has been proposed in for solving the following problem in Euclidean space:
\begin{equation}\label{MEAL_problem}
        \min_{x}  \ f(x) + g(x), \
        \st, \ A x=b,
\end{equation}
where $f$ is smooth and $g$ is weakly convex with easily computable proximal mapping. In particular, the authors of \cite{zeng2022moreau} proposed an augmented Lagrangian method for solving the Moreau envelope smoothed problem of \eqref{MEAL_problem}, inspired by the earlier work on smoothed augmented Lagrangian methods for nonconvex optimization in \cite{zhang2020proximal,zhang2022global}.

We apply the same idea of Moreau envelope smoothing and design our RADMM algorithm. Define the augmented Lagrangian function of \eqref{Moreau_envelope} as:
\begin{equation}\label{aug-Lag-func-MY-smooth}
\mathcal{L}_{\rho,\gamma} (x,z;\lambda)=f(x)+g_{\gamma}(z)+\langle\lambda, A x-z\rangle+\frac{\rho}{2}\|A x-z\|^2. %\ x\in\M,
\end{equation}
A direct application of ADMM for solving \eqref{Moreau_envelope} yields the following updating scheme:
\begin{equation}\label{R_admm_Moreau}
\begin{split}
    x^{k+1} := & \ \argmin_{x\in\M}\ \mathcal{L}_{\rho,\gamma} (x,z^k;\lambda^k) \\
    z^{k+1} := & \ \argmin_{z}\ \mathcal{L}_{\rho,\gamma} (x^{k+1},z;\lambda^k) \\
    \lambda^{k+1} := & \ \lambda^{k} + \rho (A x^{k+1}-z^{k+1}).
\end{split}
\end{equation}
Now since the $x$-subproblem in \eqref{R_admm_Moreau} is usually not easy to solve, we propose to replace it with a Riemannian gradient step, and this leads to our RADMM, which iterates as follows:
\begin{equation}\label{R_admm_Moreau-linearize}
\begin{split}
    x^{k+1} := & \ \retr_{x^k}(-\eta_k\grad_x\mathcal{L}_{\rho,\gamma} (x^k,z^k;\lambda^k)) \\
    z^{k+1} := & \ \argmin_{z}\ \mathcal{L}_{\rho,\gamma} (x^{k+1},z;\lambda^k) \\
    \lambda^{k+1} := & \ \lambda^{k} + \rho (A x^{k+1}-z^{k+1}),
\end{split}
\end{equation}
where $\eta_k>0$ is a stepsize, and $\grad_x \mathcal{L}_{\rho,\gamma}$ denotes the Riemannian gradient of $\mathcal{L}_{\rho,\gamma}$ with respect to $x$. The remaining thing is to discuss how to solve the $z$-subproblem in \eqref{R_admm_Moreau-linearize}. It turns out that it is closely related to the proximal mapping of function $g$, and can be easily solved as long as the proximal mapping of $g$ can be easily evaluated, as shown in the following lemma.

\begin{lemma}\label{lemma_z_update}
    The solution of the $z$-subproblem in \eqref{R_admm_Moreau-linearize} is given by
    \begin{equation}\label{z_update}
        z^{k+1} := \frac{\gamma}{1+\gamma\rho}\left(\frac{1}{\gamma}y^{k+1}+\lambda^{k}+\rho A x^{k+1}\right),
    \end{equation}
    where
    \begin{equation}\label{y_update}
        y^{k+1} := \mathrm{prox}_{\frac{1+\rho\gamma}{\rho}g}\left(A x^{k+1}+\frac{1}{\rho}\lambda^{k}\right),
    \end{equation}
    where $\mathrm{prox}_h$ denotes the proximal mapping of function $h$, which is defined as
    \[
        \mathrm{prox}_h(u) = \argmin_v \ h(v) + \frac{1}{2}\|u-v\|_2^2.
    \]
\end{lemma}

\begin{proof}{Proof.}
    The $z$-subproblem in \eqref{R_admm_Moreau-linearize} can be equivalently rewritten as
    \begin{equation}\label{z_step}
        (z^{k+1},y^{k+1}) := \argmin_{z,y}\ g(y)+\frac{1}{2\gamma}\|y-z\|^2 + \langle\lambda^{k}, A x^{k+1}-z\rangle+\frac{\rho}{2}\|A x^{k+1}-z\|^2.
    \end{equation}
    The optimality conditions of \eqref{z_step} are given by
    \begin{subequations}
    \begin{align}
        0 = & \frac{1}{\gamma}(z^{k+1}-y^{k+1}) - \lambda^k + \rho(z^{k+1}-Ax^{k+1}),\label{z_step-optcond-1} \\
        0 \in & \ \partial g(y^{k+1}) + \frac{1}{\gamma}(y^{k+1}-z^{k+1}). \label{z_step-optcond-2}
    \end{align}
    \end{subequations}
    It is easy to see that \eqref{z_step-optcond-1} immediately yields \eqref{z_update}. Plugging \eqref{z_update} into \eqref{z_step-optcond-2} gives
    \[
        0 \in  \ \frac{1+\gamma\rho}{\rho}\partial g(y^{k+1}) + y^{k+1} - \left(Ax^{k+1}+\frac{\lambda^k}{\rho}\right),
    \]
    which implies
    \[
        y^{k+1} = \argmin_y \ \frac{1+\gamma\rho}{\rho} g(y) + \frac{1}{2}\left\|y-\left(Ax^{k+1}+\frac{\lambda^k}{\rho}\right)\right\|_2^2,
    \]
    i.e., \eqref{y_update} holds.
\end{proof}

Our RADMM for solving \eqref{problem_nonsmooth_splitting} can therefore be summarized as in Algorithm \ref{R_admm_Moreau_grad_step}. We can see that all the steps can be easily computed and implemented.

\begin{algorithm}[ht]
   \caption{A Riemannian ADMM}
   \label{R_admm_Moreau_grad_step}
\begin{algorithmic}[1]
   \STATE {\bfseries Input:} $(x^{0},z^{0};\lambda^{0})$, stepsize $\eta_k>0$, parameters $\rho>0$ and $\gamma>0$.
   \FOR{$k=0,1,2,...$}
   \STATE Update
    $x^{k+1} := \retr_{x^k}(-\eta_k\grad_x\mathcal{L}_{\rho,\gamma} (x^k,z^k;\lambda^k))$.
   \STATE Update
    $y^{k+1} := \mathrm{prox}_{\frac{1+\rho\gamma}{\rho}g}\left(A x^{k+1}+\frac{1}{\rho}\lambda^{k}\right)$.
   \STATE Update
    $z^{k+1} := \frac{\gamma}{1+\gamma\rho}\left(\frac{1}{\gamma}y^{k+1}+\lambda^{k}+\rho A x^{k+1}\right)$.
   \STATE Update
    $\lambda^{k+1} := \lambda^{k} + \rho (A x^{k+1}-z^{k+1})$
   \ENDFOR
\end{algorithmic}
\end{algorithm}

\section{Convergence Analysis}\label{sec_convergence}

In this section, we analyze the iteration complexity of Algorithm \ref{R_admm_Moreau_grad_step} for obtaining an $\epsilon$-stationary point of \eqref{problem_nonsmooth_splitting}. The following assumption is needed in the analysis.

\begin{assumption}\label{assum}
    We assume $f$, $g$ and $\M$ in \eqref{problem_nonsmooth_splitting} satisfy the following conditions.
    \begin{enumerate}
        \item $\M\subset\mathbb{R}^n$ is a compact and complete Riemannian manifold embedded in Euclidean space $\br^n$ with diameter $D$;

        \item $\nabla f$ is Lipschitz continuous with Lipschitz constant $L$ in the ambient space $\br^n$;

        \item $g$ is convex and Lipschitz continuous with Lipschitz constant $L_g$ in the ambient space $\br^m$.
    \end{enumerate}
\end{assumption}
Also since $\M$ is compact and $\nabla f$ is continuous, we are able to denote the maximum of the norm of $f$ as a constant $M$, i.e.,
\begin{equation}\label{grad_upper_bound}
    \|\nabla f(x)\|\leq M,\ \forall x\in\M.
\end{equation}

Now we proceed to study the optimality of the problem \eqref{problem_nonsmooth_splitting}. First, we note that the first-order optimality conditions of \eqref{problem_nonsmooth_splitting} are given by (see, e.g., \cite{Yang-manifold-optimality-2014}):
\begin{equation}
\begin{split}
    0 &= \grad_x \mathcal{L}(x^*,y^*,\lambda^*) =\proj_{\T_{x^*}\M}\left(\nabla f(x^*) + A^\top\lambda^* \right), \\ %= \grad_x f(x^*) + \grad_x (\langle\lambda, Ax - y\rangle) \\
    0 &\in \partial_y \mathcal{L}(x^*,y^*,\lambda^*)=\partial g(y^*) - \lambda^*, \\
    0 &= A x^* - y^*,\\ x^* & \in \M,
\end{split}
\end{equation}
where the Lagrangian function of \eqref{problem_nonsmooth_splitting} is defined as
\[
\mathcal{L}(x,y,\lambda):= f(x) + g(y)+\langle\lambda, Ax - y\rangle.
\]
Based on this, we can define the $\epsilon$-stationary point of \eqref{problem_nonsmooth_splitting} as follows.
\begin{definition}\label{definition_eps_stationary}
    For $(x,y,\lambda)$ with $x\in\M$, denote
    $$
        \partial\mathcal{L}(x,y,\lambda):=\begin{bmatrix}
        \proj_{\T_{x}\M}\left(\nabla f(x) + A^\top\lambda \right) \\
        \partial g(y) - \lambda \\
        A x-y
        \end{bmatrix}.
    $$
    Then $(\bar{x},\bar{y},\bar{\lambda})$ with $\bar{x}\in\M$ is called an $\epsilon$-stationary point of \eqref{problem_nonsmooth_splitting} if there exists $ G\in\partial\mathcal{L}(\bar{x},\bar{y},\bar{\lambda})$ such that $\|G\|_2\leq\epsilon$.
\end{definition}

Before we present our main convergence results, we need the following lemmas. The first one is a brief recap of the properties of Moreau envelope (see e.g. \cite{beck2017first} Chapter 6).
\begin{lemma}[Properties of Moreau envelope]\label{lemma_moreau_envelope}
    Suppose $g$ is a $L_g$ Lipschitz continuous and convex function. The Moreau envelope $g_{\gamma}(z):=\min_{y} g(y) + \frac{1}{2\gamma}\|z - y\|^2$ satisfies the following:
    \begin{enumerate}
        \item $0\leq g(z) - g_{\gamma}(z) \leq \gamma L_g^2$;
        \item $\nabla g_{\gamma}(z) = \frac{1}{\gamma}(z - \prox_{\gamma g}(z))$; %where $\prox_{\gamma g}(z):=\argmin_{y} g(y) + \frac{1}{2\gamma}\|z - y\|^2$ is the proximal operator;
        \item $g_\gamma(z)$ is $L_g$-Lipschitz continuous;
        \item $g_{\gamma}(z)$ is $1/\gamma$ Lipschitz smooth, i.e. $\nabla g_{\gamma}(z)$ is Lipschitz continuous with parameter $1/\gamma$.
    \end{enumerate}
\end{lemma}

Now we proceed to bound the difference of dual sequence by the primal sequence. %, which is crucial in our final convergence result:
\begin{lemma}[Bound dual by primal]\label{bound_dual_by_primal}
    For the updates of Algorithm \ref{R_admm_Moreau_grad_step}, we have:
    \begin{equation}\label{bound_dual_by_primal-eq}
        \|\lambda^{k+1}-\lambda^{k}\|\leq\frac{1}{\gamma }\|z^{k+1}-z^{k}\|.
    \end{equation}
\end{lemma}
\begin{proof}{Proof.}
Note that the optimality conditions of the $z$-subproblem in \eqref{R_admm_Moreau-linearize} is given by:
\begin{equation}\label{Optimality_update}
    \nabla g_{\gamma }(z^{k+1}) - \lambda^{k} + \rho(z^{k+1}-A x^{k+1}) = 0,
\end{equation}
which, together with the $\lambda$ update in \eqref{R_admm_Moreau-linearize}, yields
\begin{equation}\label{grad_g_eq_lambda}
    \nabla g_{\gamma }(z^{k+1}) = \lambda^{k+1}.
\end{equation}
The desired result \eqref{bound_dual_by_primal-eq} follows from the fact that $g_{\gamma }$ is $1/\gamma$-Lipschitz smooth as in Lemma \ref{lemma_moreau_envelope}.
\end{proof}

We now provide the smoothness notion over manifolds, which is also known as Lipschitz-type gradient for pullbacks.
\begin{definition}[\cite{boumal2019global}]\label{def_geodesic_smoothness}
    Function $f$ is called $L_{1}$-retraction smooth on complete Riemannian manifold $\M$ if $\forall x\in\M$ and $\forall v\in\T_{x}\M$, it holds that
    \begin{equation}\label{geodesic-smooth-inequality}
        f(\retr_{x}(v))\leq f(x) + \langle \grad f(x), v  \rangle + \frac{L_{1}}{2}\|v\|^2.
    \end{equation}
    %Here the inner product and norm are defined by Riemannian metric tensor. We always take the metric tensor to be common Euclidean inner product unless otherwise specified.
\end{definition}

The following lemma is from \cite{boumal2019global}, which bridges the smoothness on the manifold with the smoothness in the ambient Euclidean space.

\begin{lemma}[\cite{boumal2019global}]\label{lemma_geodesic_smoothness}
    Suppose $\M\in\mathbb{E}$ is a compact and complete Riemannian manifold embedded in Euclidean space $\mathbb{E}$ and $f$ is $L$-Lipschitz smooth in $\mathbb{E}$, then $f$ is also $L_{1}$-retraction smooth, where $L_{1}$ is determined by the manifold $\M$ and $f$.
    Specifically, it can be shown~(see \cite{boumal2019global}) that there exist positive constants $\alpha$ and $\beta$ so that $\forall x\in\M$ and $\forall\eta\in \T_{x}\M$,
    \begin{equation*}
    \|\retr_{x}(\eta) - x\|\leq \alpha \|\eta\|,\text{ and }
        \|\retr_{x}(\eta) - x - \eta\|\leq \beta \|\eta\|^2.
    \end{equation*}
    As a result, it can be shown that
    \begin{equation*}
        L_1=\frac{L}{2}\alpha^2+M\beta,
    \end{equation*}
    where $M$ is the upper bound of the gradient, which is defined in \eqref{grad_upper_bound}.
\end{lemma}

Now we are ready to present the smoothness of the augmented Lagrangian function \eqref{aug-Lag-func-MY-smooth}.
\begin{lemma}\label{lemma_smoothness_lagrangian}
    For any $\{(z^k,\lambda^k)\}$ generated in Algorithm \ref{R_admm_Moreau_grad_step}, the augmented Lagrangian function $\LLg(x,z^k,\lambda^k)$ defined in \eqref{aug-Lag-func-MY-smooth} is $L_{\rho}$-retraction smooth with respect to $x\in\M$, where
    % \begin{equation}
    %     L_{\rho}= \frac{L+\rho \|A^\top A\|_{2}}{2}\alpha^2+(M + \|A\|_{2}L_g + \rho \|A^\top A\|_{2} D + \|A\|_{2}(2 L_g + \rho \|A\|_{2} D))\beta,
    % \end{equation}
    \begin{equation}\label{Lrho}
        L_\rho=\frac{L+\rho\|A\|_2^2}{2} \alpha^2+\left(M+3\|A\|_2 L_g+2 \rho\|A\|_2^2 D\right) \beta
    \end{equation}
    and $\|B\|_2$ denotes the spectral norm of matrix $B$.
    %for any $(z, \lambda)\in\{(z^k, \lambda^k)\}_{k=1,2,\ldots}$ generated in Algorithm \ref{R_admm_Moreau_grad_step}. %Note that $L_{\rho}=\mathcal{O}(\rho)$ is bounded by a linear function of $\rho$
\end{lemma}

\begin{proof}{Proof.}
    We first show that $\{z^k\}$, $\{\lambda^k\}$, $k=0,1,\ldots$, generated in Algorithm \ref{R_admm_Moreau_grad_step} are uniformly bounded. Note that from \eqref{grad_g_eq_lambda}, we have
    \begin{equation}\label{lambda-bound}
        \|\lambda^{k}\|=\|\nabla g_{\gamma }(z^{k})\|\leq L_g,
    \end{equation}
    where the inequality follows from the facts that $g$ is $L_g$-Lipschitz continuous (Assumption \ref{assum}) and Lemma \ref{lemma_moreau_envelope}. % (see, e.g. \cite{beck2017first}, Theorems 6.39 and 6.60).

    From the update of $\lambda^{k+1}$, i.e., $\lambda^{k+1} := \lambda^{k} + \rho (A x^{k+1}-z^{k+1})$, we have
    \[
    z^{k+1} = (\lambda^k - \lambda^{k+1}) / \rho + Ax^{k+1},
    \]
    which, together with  \eqref{lambda-bound}
 and Assumption \ref{assum}, immediately implies
    \begin{equation}\label{z-bound}
    \|z^{k+1}\|\leq \frac{2 L_g}{\rho} + \|A\|_{2} D.
    \end{equation}
    We now show that the gradient of $\LLg(x,z^k,\lambda^k)$, i.e., $\nabla_x\LLg(x,z^k,\lambda^k) = \nabla f(x) + A^\top\lambda^k+\rho A^\top(Ax - z^k)$, is uniformly upper bounded $\forall x\in\M$. To this end, we note that    \begin{equation}\label{proof_lipschitz_lagrangian}
    \begin{split}
        \|\nabla_x\LLg(x,z^k,\lambda^k)\|\leq &\|\nabla f(x)\| + \|A^\top \lambda^k\|+\rho \|A^\top(Ax - z^k)\| \\
        \leq & \|\nabla f(x)\| + \|A\|_{2}\|\lambda^k\|+\rho \|A^\top A\|_2 \|x\| + \rho\|A\|_{2}\|z^k\| \\
        \leq & M + \|A\|_{2}L_g + \rho \|A^\top A\|_2 D + \|A\|_{2}(2 L_g + \rho \|A\|_{2} D),
    \end{split}
    \end{equation}
    where the last inequality is due to \eqref{lambda-bound}, \eqref{z-bound} and Assumption \ref{assum}.
    
    Moreover, we have
    \begin{equation}\label{proof_lipschitz_cts_lagrangian}
    \begin{split}
        \|\nabla_x\LLg(x_1,z^k,\lambda^k) - \nabla_x\LLg(x_2,z^k,\lambda^k)\|\leq &\|\nabla f(x_1) - \nabla f(x_2)\|+\rho \|A^\top A(x_1 - x_2)\| \\
        \leq & L\|x_1 - x_2\|+\rho \|A^\top A\|_{2}\|x_1 - x_2\|.
    \end{split}
    \end{equation}
    
   By applying Lemma \ref{lemma_geodesic_smoothness} together with  \eqref{proof_lipschitz_lagrangian} and \eqref{proof_lipschitz_cts_lagrangian}, we immediately obtain the desired result.
\end{proof}

Now we give the following lemma regarding the decrease of the augmented Lagrangian function $\LLg$.
\begin{lemma}\label{lemma_decreasing_Lagrangian}
    For the sequence $\{(x^k,z^k,\lambda^k)\}$ generated in Algorithm \ref{R_admm_Moreau_grad_step}, we have:
    \begin{equation}\label{lemma_decreasing_Lagrangian-eq}
    \begin{split}
        &\mathcal{L}_{\rho ,\gamma }(x^{k+1},z^{k+1},\lambda^{k+1}) - \mathcal{L}_{\rho ,\gamma }(x^{k},z^{k},\lambda^{k}) \\
        \leq& \left(\frac{1}{\rho \gamma ^2}-\frac{\rho}{2}\right)\|z^{k+1}-z^{k}\|^2-\left(\frac{1}{\eta_{k}}-\frac{ L_{\rho }}{2}\right)\|\retr_{x^{k}}^{-1}(x^{k+1})\|^2,
    \end{split}
    \end{equation}
    where $L_{\rho }$ is defined in \eqref{Lrho}.
\end{lemma}

\begin{proof}{Proof.}
  First, we have
\begin{equation}\label{decrease-1}
\begin{split}
    & \ \mathcal{L}_{\rho ,\gamma }(x^{k+1},z^{k+1},\lambda^{k+1}) - \mathcal{L}_{\rho ,\gamma }(x^{k+1},z^{k+1},\lambda^{k}) \\ = & \ \langle \lambda^{k+1} - \lambda^{k}, A x^{k+1} - z^{k+1} \rangle \\
    = & \ \frac{1}{\rho}\|\lambda^{k+1} - \lambda^{k}\|^2 \leq \ \frac{1}{\rho\gamma ^2}\|z^{k+1}-z^{k}\|^2,
\end{split}
\end{equation}
where the inequality is from Lemma \ref{bound_dual_by_primal}.

Second, we have,
\begin{equation}\label{decrease-2}
\begin{split}
    &\mathcal{L}_{\rho ,\gamma }(x^{k+1},z^{k+1},\lambda^{k}) - \mathcal{L}_{\rho ,\gamma }(x^{k+1},z^{k},\lambda^{k})\\
    = & g_{\gamma }(z^{k+1}) - g_{\gamma }(z^{k}) + \langle \lambda^{k}, z^k - z^{k+1} \rangle + \frac{\rho}{2}(\|A x^{k+1}-z^{k+1}\|^2 - \|A x^{k+1}-z^{k}\|^2) \\
    = & g_{\gamma }(z^{k+1}) - g_{\gamma }(z^{k}) + \langle \lambda^{k}+\rho(A x^{k+1}-z^{k+1}), z^k - z^{k+1} \rangle - \frac{\rho}{2}\|z^{k+1}-z^{k}\|^2 \\
    \leq & - \frac{\rho}{2}\|z^{k+1}-z^{k}\|^2,
\end{split}
\end{equation}
where the inequality is by convexity of $g_{\gamma }$ and $\nabla g_{\gamma }(z^{k+1}) = \lambda^{k+1}=\lambda^{k}+\rho(A x^{k+1}-z^{k+1})$.

Third, by Lemma \ref{lemma_smoothness_lagrangian} and \eqref{geodesic-smooth-inequality}, we obtain
\begin{equation}\label{decrease-3}
\begin{split}
    & \ \mathcal{L}_{\rho ,\gamma }(x^{k+1},z^{k},\lambda^{k}) - \mathcal{L}_{\rho ,\gamma }(x^{k},z^{k},\lambda^{k}) \\
    \leq & \ \langle \grad_x \mathcal{L}_{\rho ,\gamma }(x^{k},z^{k},\lambda^{k}), \retr_{x^{k}}^{-1}(x^{k+1}) \rangle+\frac{L_\rho}{2}\|\retr_{x^{k}}^{-1}(x^{k+1})\|^2 \\
    = & \ -\left(\frac{1}{\eta_{k}}-\frac{ L_{\rho }}{2}\right)\|\retr_{x^{k}}^{-1}(x^{k+1})\|^2,
\end{split}
\end{equation}
where the equality follows from the $x$-update in Algorithm \ref{R_admm_Moreau_grad_step}.

Combining \eqref{decrease-1}, \eqref{decrease-2}, and \eqref{decrease-3} yields the desired result \eqref{lemma_decreasing_Lagrangian-eq}.
% Thus we get
% \[
% \begin{split}
%     &\mathcal{L}_{\rho ,\gamma }(x^{k+1},z^{k+1},\lambda^{k+1}) - \mathcal{L}_{\rho ,\gamma }(x^{k},z^{k},\lambda^{k})\\
%     \leq& (\frac{1}{\rho \gamma ^2}-\frac{\rho}{2})\|z^{k+1}-z^{k}\|^2-(\frac{1}{\eta_{k}}-\frac{ L_{\rho }}{2})\|\retr_{x^{k}}^{-1}(x^{k+1})\|^2
% \end{split}
% \]
\end{proof}

The following lemma shows that the augmented Lagrangian function $\LLg$ is lower bounded.
\begin{lemma}\label{lemma_lower_bdd}
    If $\rho\gamma\geq 1$, then the sequence $\{\mathcal{L}_{\rho,\gamma}(x^{k},z^{k},\lambda^{k})\}$ is uniformly lower bounded by $F^*-\gamma L_g^2$, where $F^*$ is the optimal value of \eqref{problem_nonsmooth}.
\end{lemma}
\begin{proof}{Proof.}
    By the $1/\gamma$ Lipschitz smoothness of $g_{\gamma}$ (see Lemma \ref{lemma_moreau_envelope}) and $\nabla g_{\gamma}(z^{k}) = \lambda^{k}$, we get
    $$
        g_{\gamma}(A x)\leq g_{\gamma}(z) + \langle \nabla g_{\gamma}(z), A x - z \rangle + \frac{1}{2\gamma}\|A x - z\|^2,
    $$
    which implies
    \begin{equation*}
    \begin{aligned}
    \mathcal{L}_{\rho,\gamma}(x^{k},z^{k},\lambda^{k})
    & = f(x^{k})+g_{\gamma}(z^{k})+\langle \lambda^{k}, A x^{k}-z^{k} \rangle+\frac{\rho}{2}\|A x^{k}-z^{k}\|^2\\
    & \geq  f(x^{k})+g_{\gamma}(A x^{k})+\left(\frac{\rho}{2} - \frac{1}{2\gamma}\right)\|A x^{k}-z^{k}\|^2 \\ & \geq f(x^{k})+g_{\gamma}(A x^{k}) \\ & \geq f(x^{k})+g(A x^{k}) - \gamma L_g^2 \\ & \geq F^* - \gamma L_g^2,
    \end{aligned}
    \end{equation*}
    where the third inequality follows from Lemma \ref{lemma_moreau_envelope}.
\end{proof}

The following lemma gives an upper bound for $G^k\in\partial\mathcal{L}(x^k,y^k,\lambda^k)$.
\begin{lemma}\label{lemma_subgradient_bound}
    Denote the iterates of Algorithm \ref{R_admm_Moreau_grad_step} by $\{(x^k, y^k,z^k,\lambda^k)\}$. There exists $G^k\in\partial\mathcal{L}(x^k,y^k,\lambda^k)$, $\forall k\geq 1$, as defined in Definition \ref{definition_eps_stationary}, such that:
    $$
        \|G^{k}\|^2\leq \frac{2}{\eta_{k}^2}\|\retr_{x^{k}}^{-1}(x^{k+1})\|^2 + \frac{2(\rho^2\|A\|_{2}^2 + 1)}{\rho^2\gamma^2} \|z^{k}-z^{k-1}\|^2 + 2\gamma^2  L_g^2.
    $$
\end{lemma}

\begin{proof}{Proof.}
From \eqref{z_step-optcond-1}, \eqref{z_step-optcond-2} and the update of $\lambda^{k+1}$ in Algorithm
\ref{R_admm_Moreau_grad_step}, we know that $\lambda^k\in\partial g(y^k)$ for $k=1,2,\ldots$. Therefore, there exist $G^k\in\partial\mathcal{L}(x^k,y^k,\lambda^k)$ such that
\begin{equation*}
\begin{aligned}
    \|G^k\|^2 = & \|\proj_{\T_{x^{k}}\M}\left(\nabla f(x^{k}) + A^\top \lambda^{k} \right)\|^2 + \| A x^{k} - y^{k} \|^2 \\
    \leq & \|\proj_{\T_{x^{k}}\M}\left(\nabla f(x^{k}) + A^\top \lambda^{k} \right)\|^2 +2 \| A x^{k} - z^{k} \|^2 + 2\|z^{k} - y^{k}\|^2.
\end{aligned}
\end{equation*}
Now from the $x$ update of Algorithm \ref{R_admm_Moreau_grad_step}, we know that
\[
\proj_{\T_{x^{k}}\M}\left(\nabla f(x^{k}) + A^\top \lambda^{k} \right)=-\frac{1}{\eta_{k}}\retr_{x^{k}}^{-1}(x^{k+1}) - \proj_{\T_{x^{k}}\M}(\rho A^\top(A x^{k}-z^{k})).
\]
Therefore, we have
\begin{equation*}
\begin{aligned}
    \|G^k\|^2 \leq & \left\|\frac{1}{\eta_{k}}\retr_{x^{k}}^{-1}(x^{k+1}) + \proj_{\T_{x^{k}}\M}(\rho A^\top(A x^{k}-z^{k}))\right\|^2 +2 \| A x^{k} - z^{k} \|^2 + 2\|z^{k} - y^{k}\|^2 \\
    \leq & \frac{2}{\eta_{k}^2}\|\retr_{x^{k}}^{-1}(x^{k+1})\|^2 + 2\rho^2 \| \proj_{\T_{x^{k}}\M}(A^\top(A x^{k}-z^{k}))\|^2 +2 \| A x^{k} - z^{k} \|^2 + 2\|z^{k} - y^{k}\|^2 \\
    \leq & \frac{2}{\eta_{k}^2}\|\retr_{x^{k}}^{-1}(x^{k+1})\|^2 + 2\rho^2\|A\|_{2}^2\|A x^{k}-z^{k}\|^2 +2 \| A x^{k} - z^{k} \|^2 + 2\|z^{k} - y^{k}\|^2 \\
    = & \frac{2}{\eta_{k}^2}\|\retr_{x^{k}}^{-1}(x^{k+1})\|^2 + 2(\rho^2\|A\|_{2}^2 + 1) \| A x^{k} - z^{k} \|^2 + 2\|z^{k} - y^{k}\|^2.
\end{aligned}
\end{equation*}
%where the last inequality is by the fact that projection operators are contractions.

Now by the update of $\lambda^k$ in Algorithm \ref{R_admm_Moreau_grad_step} and \eqref{bound_dual_by_primal-eq} we have $\rho\|A x^{k}-z^{k}\|=\|\lambda^{k} - \lambda^{k-1}\|\leq\frac{1}{\gamma }\|z^{k}-z^{k-1}\|$. By \eqref{z_step-optcond-2} we have $z^{k} - y^{k}\in \gamma\partial g(y^{k})$ so that $\|z^{k} - y^{k}\|\leq \gamma  L_g$. Combining these results we get
\begin{equation*}
\begin{aligned}
    \|G^k\|^2 \leq & \frac{2}{\eta_{k}^2}\|\retr_{x^{k}}^{-1}(x^{k+1})\|^2 + 2(\rho^2\|A\|_{2}^2 + 1) \| A x^{k} - z^{k} \|^2 + 2\|z^{k} - y^{k}\|^2 \\
    \leq & \frac{2}{\eta_{k}^2}\|\retr_{x^{k}}^{-1}(x^{k+1})\|^2 + \frac{2(\rho^2\|A\|_{2}^2 + 1)}{\rho^2\gamma^2} \|z^{k}-z^{k-1}\|^2 + 2\gamma^2  L_g^2,
\end{aligned}
\end{equation*}
which gives the desired result.
\end{proof}

Finally, we have the following convergence result for Algorithm \ref{R_admm_Moreau_grad_step}.

\begin{theorem}\label{thm_convergence}
    Denote the iterates of Algorithm \ref{R_admm_Moreau_grad_step} by $\{(x^k, y^k,z^k,\lambda^k)\}$. For a given tolerance $\epsilon>0$, we set $\rho=1/\epsilon$, $\gamma=\sqrt{\frac{2}{\rho^2} + \frac{\rho^2\|A\|_{2}^2 + 1}{\rho^3 L_{\rho}}} = \mathcal{O}(\epsilon)$, also $\eta_{k}=\eta=\frac{1}{L_{\rho}}$. Note that our choices of $\rho$ and $\gamma$ guarantees that $\rho\gamma>1$. Then there exist $G^k\in\partial\mathcal{L}(x^k,y^k,\lambda^k)$, $k=1,2,\ldots$, such that $$\min_{k=1,...,K}\|G^k\|^2\leq\epsilon^2,$$
    provided that
    $$
        K = \mathcal{O}\left(\frac{1}{\epsilon^4}\right).
    $$
    That is, Algorithm \ref{R_admm_Moreau_grad_step} generates an $\epsilon$-stationary point to \eqref{problem_nonsmooth_splitting} in $\mathcal{O}({\epsilon^{-4}})$ iterations.
\end{theorem}
\begin{proof}{Proof.}
    From Lemma \ref{lemma_subgradient_bound}, there exist $G^k\in\partial\mathcal{L}(x^k,y^k,\lambda^k)$, $k=1,2,\ldots$ such that
    $$
        \|G^{k}\|^2\leq \frac{2}{\eta_{k}^2}\|\retr_{x^{k}}^{-1}(x^{k+1})\| + \frac{2(\rho^2\|A\|_{2}^2 + 1)}{\rho^2\gamma^2} \|z^{k}-z^{k-1}\|^2 + 2\gamma^2  L_g^2,
    $$
    which, combining with \eqref{lemma_decreasing_Lagrangian-eq} and $\eta_k=1/L_\rho$, yields
    \[
    \begin{split}
        \|G^{k}\|^2& \ \leq \frac{4}{\eta_{k}}\left(\mathcal{L}_{\rho ,\gamma }(x^{k},z^{k},\lambda^{k}) - \mathcal{L}_{\rho ,\gamma }(x^{k+1},z^{k+1},\lambda^{k+1})\right) \\
        & \ + \left(\frac{2(\rho^2\|A\|_{2}^2 + 1)}{\rho^2\gamma^2}\|z^{k}-z^{k-1}\|^2 - \frac{4}{\eta_{k}}\left(\frac{\rho}{2} - \frac{1}{\rho \gamma ^2}\right)\|z^{k}-z^{k+1}\|^2 \right) + 2\gamma ^2 L_g^2.
    \end{split}
    \]
     Now by taking $\gamma$, $\rho$ and $\eta_{k}=\eta$ as described in the theorem, it is easy to verify that %\footnote{$\frac{2(\rho^2\|A\|_{op}^2 + 1)}{\rho^2\gamma^2} \leq \frac{4}{\eta_{k}}\left(\frac{\rho}{2} - \frac{1}{\rho \gamma ^2}\right)$ is equivalent to $\rho^2\gamma^2\geq 2 + \frac{\rho^2\|A\|_{2}^2 + 1}{\rho L_{\rho}}$}
    $$
        \frac{2(\rho^2\|A\|_{2}^2 + 1)}{\rho^2\gamma^2} \leq \frac{4}{\eta_{k}}\left(\frac{\rho}{2} - \frac{1}{\rho \gamma ^2}\right).
    $$
    Therefore, we have
    \[
    \begin{split}
        \|G^{k}\|^2&\leq \frac{4}{\eta_{k}}\left(\mathcal{L}_{\rho ,\gamma }(x^{k},z^{k},\lambda^{k}) - \mathcal{L}_{\rho ,\gamma }(x^{k+1},z^{k+1},\lambda^{k+1})\right) \\
        & \ + \left(\frac{4}{\eta_{k}}\left(\frac{\rho}{2} - \frac{1}{\rho \gamma ^2}\right)\|z^{k}-z^{k-1}\|^2 - \frac{4}{\eta_{k}}\left(\frac{\rho}{2} - \frac{1}{\rho \gamma ^2}\right)\|z^{k}-z^{k+1}\|^2 \right) + 2\gamma ^2 L_g^2.
    \end{split}
    \]
    Now by summing this inequality over $k=1,\ldots,K$ and using Lemma \ref{lemma_lower_bdd}, we get
    \[
    \begin{split}
        \frac{1}{K}\sum_{k=1}^K\|G^{k}\|^2&\leq \frac{4}{\eta K}(\mathcal{L}_{\rho ,\gamma}(x^{1},z^{1},\lambda^{1}) - F^* + \gamma L_g^2)+\frac{2\rho}{\eta K}\|z^{1}-z^{0}\|^2+ 2\gamma^2 L_g^2.
    \end{split}
    \]

    Since we take $\gamma=\mathcal{O}(\epsilon)$, $\rho=\frac{1}{\epsilon}$ and $\eta=\frac{1}{L_{\rho}}=\mathcal{O}({\epsilon})$, to ensure $\min_{k=1,...,K}\|G^k\|^2\leq\epsilon^2$, we need $K=\mathcal{O}(\frac{1}{\epsilon^4})$.
\end{proof}

\section{Applications and Numerical Experiments}\label{sec_numeric}

Problem \eqref{problem_nonsmooth} finds many applications in machine learning, statistics and signal processing. For example, K-means clustering \cite{carson2017manifold}, sparse spectral clustering \cite{lu2018nonconvex,park2018spectral}, and orthogonal dictionary learning \cite{Spielman-Wang-Wright-2012,Demanet-Hand-2014,Qu-Sun-Wright-sparse-vector-2016,Sun-CDR-part1-2017,Sun-CDR-part2-2017} are all of the form of \eqref{problem_nonsmooth}.
In this section, we present two representative applications of \eqref{problem_nonsmooth} and then report the numerical results of our Algorithm \ref{R_admm_Moreau_grad_step} for solving them.

{\bf Example 1. Sparse Principal Component Analysis (PCA).}
    Principal Component Analysis, proposed by Pearson \cite{pearson1901liii} and later developed by Hotelling \cite{hotelling1933analysis}, is one of the most fundamental statistical tools in analyzing high-dimensional data. Sparse PCA seeks principal components with very few nonzero components. For given data matrix $A\in\br^{m\times n}$, the sparse PCA that seeks the leading $p$ $(p<\min\{m,n\})$ sparse loading vectors can be formulated as
    \begin{equation}\label{sPCA}
    \begin{split}
            \min_{X} & \ F(X):=-\frac{1}{2}\mathrm{Tr}(X^\top A^\top A X) + \mu\|X\|_{1} \\
            \st & \ X\in\St(n,p),
    \end{split}
    \end{equation}
    where $\mu>0$ is a weighting parameter. This is the original formulation of sparse PCA as proposed by Jolliffe \etal in \cite{Jolliffe2003}, where the model is called SCoTLASS and imposes sparsity and orthogonality to the loading vectors simultaneously. When $\mu=0$, \eqref{sPCA} reduces to computing the leading $p$ eigenvalues and the corresponding eigenvectors of $A^\top A$. When $\mu>0$, the $\ell_1$ norm $\|X\|_1$ can promote sparsity of the loading vectors. There are many numerical algorithms for solving \eqref{sPCA} when $p=1$. In this case, \eqref{sPCA} is relatively easy to solve because $X$ reduces to a vector and the constraint set reduces to a sphere. However, there has been very limited literature for the case $p>1$. Existing works, including \cite{Zou-spca-2006,daspremont-sparsePCA-direct-formulation-2007,Shen-Huang-spca-2008,Journee-Nesterov-sparsePCA-JMLR-2010,Ma-SPCA-2011-submit}, do not impose orthogonal loading directions. As discussed in \cite{Journee-Nesterov-sparsePCA-JMLR-2010}, ``Simultaneously enforcing sparsity and orthogonality seems to be a hard (and perhaps questionable) task.'' We refer the interested reader to \cite{zou2018selective} for more details on existing algorithms for solving sparse PCA.

{\bf Example 2. Orthogonal Dictionary Learning (ODL) and Dual principal component pursuit (DPCP).}
In ODL, one is given a set of $p$ $(p\gg n)$ data points $\by_1,\ldots,\by_p \in \br^n$ and aims to find an orthonormal basis of $\br^n$ to represent them compactly. In other words, by letting $Y = [\by_1,\ldots,\by_p]\in\br^{n\times p}$, we want to find an orthogonal matrix $X\in\br^{n\times n}$ and a sparse matrix $A\in\br^{n\times p}$ such that $Y = XA$. Since $X$ is orthogonal, we know that $A = X^\top Y$. This naturally leads to the following matrix version of ODL \cite{Spielman-Wang-Wright-2012,Demanet-Hand-2014,Qu-Sun-Wright-sparse-vector-2016,Sun-CDR-part1-2017,Sun-CDR-part2-2017}:
\begin{equation}\label{dpcp}
    \begin{split}
            \min_{X} & \ \|Y^\top X\|_1 \\
            \st & \ X\in\St(n,n).
    \end{split}
    \end{equation}
Here, the $\ell_1$ norm is used to promote the sparsity of $A=X^\top Y$, and the constraint set $\St(n,n)$ is known as the orthogonal group, which is a special case of the Stiefel manifold.

Another representative application of \eqref{dpcp} is robust subspace recovery (RSR) \cite{Lerman-convex-relax-FoCM-2015,Lerman-RSR-IRLS-2018,Lerman-well-tempered-2019,Lerman-stochastic-private-2022}.
RSR aims to fit a linear subspace to a dataset corrupted by outliers, which is a fundamental problem in machine learning and data mining. RSR can be described as follows. Given a dataset $Y = [\X, \CO]\Gamma\in\br^{n\times (p_1+p_2)}$, where $\X\in\br^{n\times p_1}$ are inlier points spanning a $d$-dimensional subspace $\CS$ of $\br^n$ ($d<p_1$), $\CO\in\br^{n\times p_2}$ are outlier points without linear structure, and $\Gamma\in\br^{(p_1+p_2)\times(p_1+p_2)}$ is an unknown permutation, the goal is to recover the inlier space $\CS$, or equivalently, to cluster the points into inliers and outliers. For a more comprehensive review of RSR, see the recent survey paper by Lerman and Maunu \cite{Lerman-PIEEE-survey}. The dual principal component pursuit (DPCP) is a recently proposed approach to RSR that seeks to learn recursively a basis for the orthogonal complement $\CS$ by solving~\eqref{dpcp} when $X$ reduces to a vector, i.e.,
\be\label{DPCP}
\min_{\bx\in\br^n} \ f(\bx):=\|Y^\top \bx\|_1 \quad \st \quad \|\bx\|_2 = 1,
\ee
The idea of DPCP is to first compute a normal vector $\bx$ to a hyperplane $\CH$ that contains all inliers $\X$. As outliers are not orthogonal to $\bx$ and the number of outliers is known to be small, the normal vector $\bx$ can be found by solving \eqref{DPCP}. It is shown in \cite{Tsakiris-Vidal-2018,zhu2018dual} that under certain conditions, solving \eqref{DPCP} indeed yields a vector that is orthogonal to $\CS$, given that the number of outliers $p_2$ is at most on the order of $O(p_1^2)$.
If $d$ is known, then one can recover $\CS$ as the intersection of the $p:=n-d$ orthogonal hyperplanes that contain $\X$, which amounts to solving the following matrix optimization problem:
\be\label{DPCP-matrix}
    \min_{X\in\R^{n\times(n-d)}} \|Y^\top X\|_1\quad \st\quad X^\top X = I_{n-d}.
\ee
Note that \eqref{sPCA}-\eqref{DPCP-matrix} are all in the form of \eqref{problem_nonsmooth}.

\subsection{Numerical Experiments on Sparse PCA}

In this subsection, we conduct experiments to test the performance of our Riemannian ADMM\footnote{Our code is available at \url{https://github.com/JasonJiaxiangLi/RADMM}.} for solving sparse PCA \eqref{sPCA}, and compare it with the performance of ManPG \cite{chen2020proximal} and Riemannian subgradient method \cite{ferreira1998subgradient,li2019weakly}. To apply Riemannian ADMM, we first rewrite \eqref{sPCA} as:
\begin{equation}\label{sPCA-split}
\begin{split}
        \min_{X,Y} & \ -\frac{1}{2}\Tr(X^\top A^\top A X)+\mu\|Y\|_{1} \\
        \st & \ X=Y,\ X\in\St(n,p).
\end{split}
\end{equation}
Now we see that the nonsmooth function $\|\cdot\|_1$ and the manifold constraint are associated with different variables. Thus, the two difficult terms are separated. Using the Moreau envelope smoothing, the smoothed problem of \eqref{sPCA-split} is given by:
\begin{equation}\label{sPCA-split-smooth}
\begin{split}
        \min_{X,Z} & \ -\frac{1}{2}\Tr(X^\top A^\top A X)+ g_\gamma(Z) \\
        \st & \ X=Z,\ X\in\St(n,p),
\end{split}
\end{equation}
where $g_\gamma(Z):=\min_Y\{\mu\|Y\|_1+\frac{1}{2\gamma}\|Y-Z\|_F^2\}$. The augmented Lagrangian function of \eqref{sPCA-split-smooth} is given by
\[
\mathcal{L}_{\rho,\gamma}(X,Z;\Lambda) = -\frac{1}{2}\Tr(X^\top A^\top A X)+ g_\gamma(Z) + \langle\Lambda,X-Z\rangle+\frac{\rho}{2}\|X-Z\|_F^2.
\]
Therefore, one iteration of our Riemannian ADMM (Algorithm \ref{R_admm_Moreau_grad_step}) for solving \eqref{sPCA-split} reduces to:
\begin{equation}\label{R_admm_Moreau_spca}
\begin{split}
    X^{k+1} := & \ \retr_{X^{k}}(-\eta_{k}\proj_{\T_{X^{k}}\St(n,p)}(-A^\top A X^k+\Lambda^{k}+\rho(X^{k}-Z^{k}))) \\
    Y^{k+1} := & \ \mathrm{prox}_{\frac{\mu(1+\rho\gamma)}{\rho}\|\cdot\|_1}\left(X^{k+1}+\frac{1}{\rho}\Lambda^k\right)\\
    Z^{k+1} := & \ \frac{\gamma}{1 +\gamma\rho}\left(\frac{1}{\gamma }Y^{k+1}+\Lambda^k+\rho X^{k+1}\right) \\
    \Lambda^{k+1} := & \ \Lambda^{k} + \rho (X^{k+1}-Z^{k+1}).
\end{split}
\end{equation}

The ManPG \cite{chen2020proximal} for solving \eqref{sPCA} updates the iterates as follows:
\begin{equation}\label{manpg_spca}
\begin{split}
        V^k  & :=\argmin_{V\in \T_{X^{k}}\St(n, p)} \langle -A^\top A X^k, V\rangle + \frac{1}{2 t}\|V\|^2+\mu\|X^k + V\|_1\\
        X^{k+1} & := \retr_{X^k}(\alpha V^k),
\end{split}
\end{equation}
where $\alpha$ and $t$ are stepsizes. The authors of \cite{chen2020proximal} suggest to solve the $V$ subproblem by using a semi-smooth Newton method. The Riemannian subgradient method (RSG) \cite{ferreira1998subgradient} for solving \eqref{sPCA} updates the iterates as follows:
\begin{equation}\label{rgrad_spca}
    X^{k+1}:= \retr_{X^k}(-\eta_k\proj_{\T_{X^{k}}\St(n, p)}(-A^\top A X^{k} + \mu D^{k})),\ \mbox{ with } D^{k}\in \partial \|X^{k}\|_1.
\end{equation}

We now describe the setup of our numerical experiment. The data matrix $A\in\R^{m\times n}$ is generated randomly whose entries follow the standard Gaussian distribution. We choose $\mu$ from $\{0.5,0.7,1\}$, $n$ from $\{100, 300, 500\}$ with $m=n$, and $p$ from $\{50, 100\}$. In our Riemannian ADMM, we set $\gamma=10^{-8}$, $\rho=10^{2}$ and $\eta_{k}=\eta=10^{-2}$. The code of ManPG is downloaded from the authors' website of \cite{chen2020proximal} and default settings of the parameters are used. In RSG \eqref{rgrad_spca}, we set the stepsize $\eta_{k}=\eta=10^{-2}$ as a result of a simple grid search. For all three algorithms, we terminate them when the change of the objective function in two consecutive iterations is smaller than $10^{-8}$,
which means 
\[
|F(X^{k+1})-F(X^{k})|<10^{-8}
\]
for ManPG \eqref{manpg_spca}, RSG \eqref{rgrad_spca} and RADMM \eqref{R_admm_Moreau_spca}, where $F(X):=-\frac{1}{2}\Tr(X^\top A^\top A X)+\mu\|X\|_{1}$. Moreover, we also terminate the three algorithms when the maximal iteration number, which is set $1000$, is reached. For different combinations of $\mu$, $n$ and $p$, we report the objective value ``obj'' ($F(X^k)$ for ManPG and RSG, and $F(Y^k)$ for RADMM), CPU time and the sparsity of the solution ``Spa'' in Table \ref{table_spca_sparsity}. Here the ``sparsity'' is the percentage of the zero entries of the iterate ($X^k$ for ManPG and RSG, and $Y^k$ for RADMM). Moreover, note that $Y^k$ in RADMM \eqref{R_admm_Moreau_spca} is not on the Stiefel manifold, we thus report the constraint violation ``infeas'', which is defined as $\|(Y^k)^\top Y_k - I_p\|_{F}$, in Table \ref{table_spca_sparsity} for RADMM. From Table \ref{table_spca_sparsity} we have the following observations: (i) both ManPG and RADMM generated sparse solutions especially when $\mu$ is large, while RSG is not capable of generating sparse solutions; (ii) RSG is very slow. It cannot decrease the objective value to the same level as ManPG and RADMM; (iii) RADMM is always faster than ManPG, sometimes is about 10 to 20 times faster. (iv) In most cases, RADMM yields iterates with better objective value than ManPG, and although $Y^k$ generated by RADMM is not on the Stiefel manifold, the constraint violation is small -- usually in the order of $10^{-6}$-$10^{-8}$.

\begin{table}[ht]
%\vskip 0.15in
\begin{center}
\begin{small}
\begin{tabular}{c  c | c  c  c | c  c  c | c  c  c  c }
 \hline
 \multicolumn{2}{c|}{Settings} & \multicolumn{3}{c|}{RSG} & \multicolumn{3}{c|}{ManPG} & \multicolumn{4}{c}{RADMM} \\
 \hline
 $\mu$ & $(n, p)$ & obj & CPU & Spa & obj & CPU & Spa & obj & CPU & Spa & infeas \\
 \hline
 \multirow{4}{*}{0.5} & $(300, 50)$ & 23.9783 & 0.5725 & 0 & 6.1015 & 1.6808 & 0.9964 & 6.0794 & 0.3550 & 0.9965 & 1.14e-6 \\
  & $(300, 100)$ & 44.9207 & 1.4091 & 0 & 9.9683 & 16.9343 & 0.9966 & 9.4524 & 1.0113 & 0.9964 & 4.43e-6 \\
  & $(500, 50)$ & 34.8607 & 1.1545 & 0 & 4.8868 & 1.7355 & 0.9977 & 4.7141 & 0.8379 & 0.9980 & 7.07e-8 \\
  & $(500, 100)$ & 72.1180 & 2.2447 & 0 & 12.0830 & 15.4234 & 0.9980 & 11.7489 & 1.5738 & 0.9980 & 1.00e-7 \\
 \hline
 \multirow{4}{*}{0.7} & $(300, 50)$ & 50.0266 & 0.5584 & 0 & 14.9053 & 1.7990 & 0.9965 & 14.9497 & 0.2860 & 0.9967 & 9.90e-8 \\
  & $(300, 100)$ & 99.1306 & 1.4196 & 0 & 29.0171 & 16.7438 & 0.9966 & 28.9101 & 0.8185 & 0.9967 & 1.40e-7 \\
  & $(500, 50)$ & 73.4292 & 1.1515 & 0 & 14.3927 & 1.9293 & 0.9978 & 14.2181 & 0.7760 & 0.9980 & 9.90e-8 \\
  & $(500, 100)$ & 147.0228 & 2.2224 & 0 & 29.8765 & 16.9296 & 0.9980 & 29.6908 & 1.2075 & 0.9980 & 1.40e-7 \\
 \hline
 \multirow{4}{*}{1.0} & $(300, 50)$ & 99.5018 & 0.5593 & 0 & 29.4374 & 2.2295 & 0.9967 & 29.6217 & 0.1879 & 0.9967 & 1.41e-7 \\
  & $(300, 100)$ & 202.9473 & 1.4154 & 0 & 61.5334 & 16.0349 & 0.9965 & 61.0310 & 0.5699 & 0.9967 & 2.00e-7 \\
  & $(500, 50)$ & 149.1125 & 1.1564 & 0 & 30.5119 & 1.8004 & 0.9980 & 30.4099 & 0.4336 & 0.9980 & 1.41e-7 \\
  & $(500, 100)$ & 295.5895 & 2.2384 & 0 & 59.5210 & 18.3017 & 0.9980 & 59.5309 & 1.0377 & 0.9980 & 2.00e-7 \\
 \hline
\end{tabular}
\end{small}
\end{center}
\caption{Comparison of RSG \eqref{rgrad_spca}, ManPG \eqref{manpg_spca}, and RADMM \eqref{R_admm_Moreau_spca} for solving \eqref{sPCA}. The results are averaged for 10 repeated experiments with random initializations.}
\label{table_spca_sparsity}
%\vskip -0.1in
\end{table}

To better illustrate the behavior of the three algorithms, we further draw some figures in Figure \ref{fig:spca_cpu_time}, to show how the objective function value decreases along with the CPU time. Here and in all figures presented later, we use the minimum function value returned by all tested algorithms as $f^*$. From Figure \ref{fig:spca_cpu_time} we can clearly see that RGS quickly stops decreasing the objective value, while both ManPG and RADMM can decrease the objective value to a much lower level. Moreover, RADMM is much faster than ManPG.

\begin{figure}[t!]
\begin{center}
\subfigure[$(n, p)=(300,50)$]{\includegraphics[clip, trim=1.5cm 6cm 2cm 6cm,width=0.45\columnwidth]{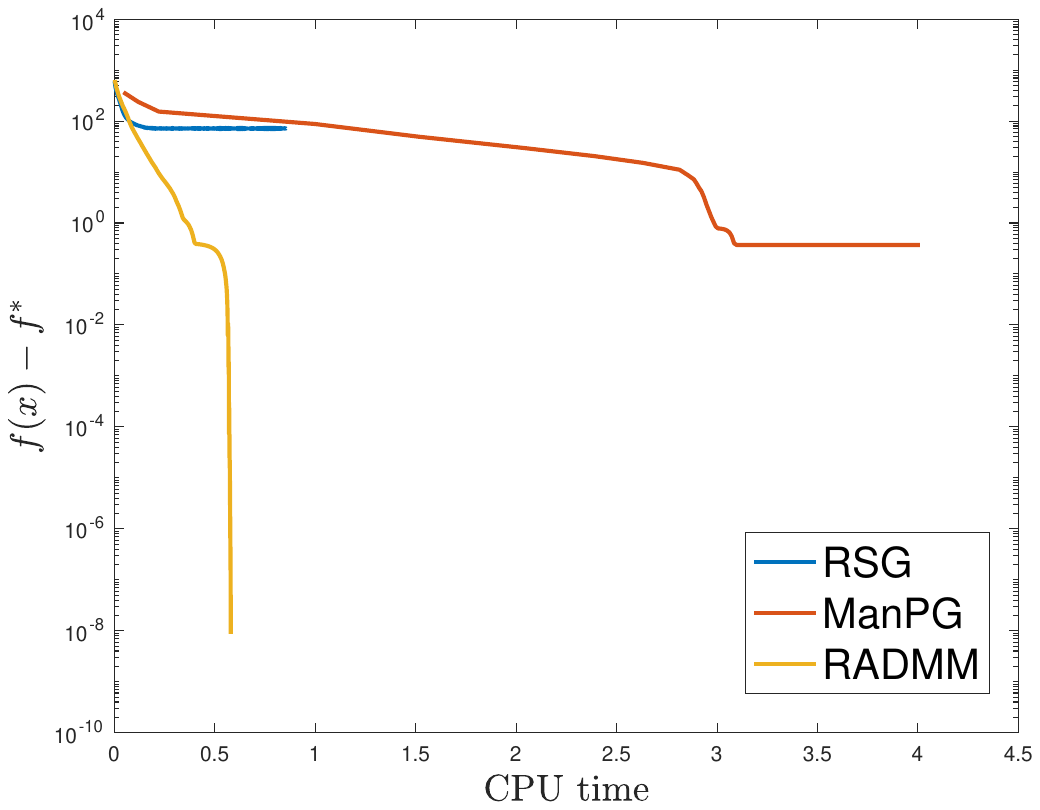}}
\subfigure[$(n,p)=(300,100)$]{\includegraphics[clip, trim=1.5cm 6cm 2cm 6cm,width=0.45\columnwidth]{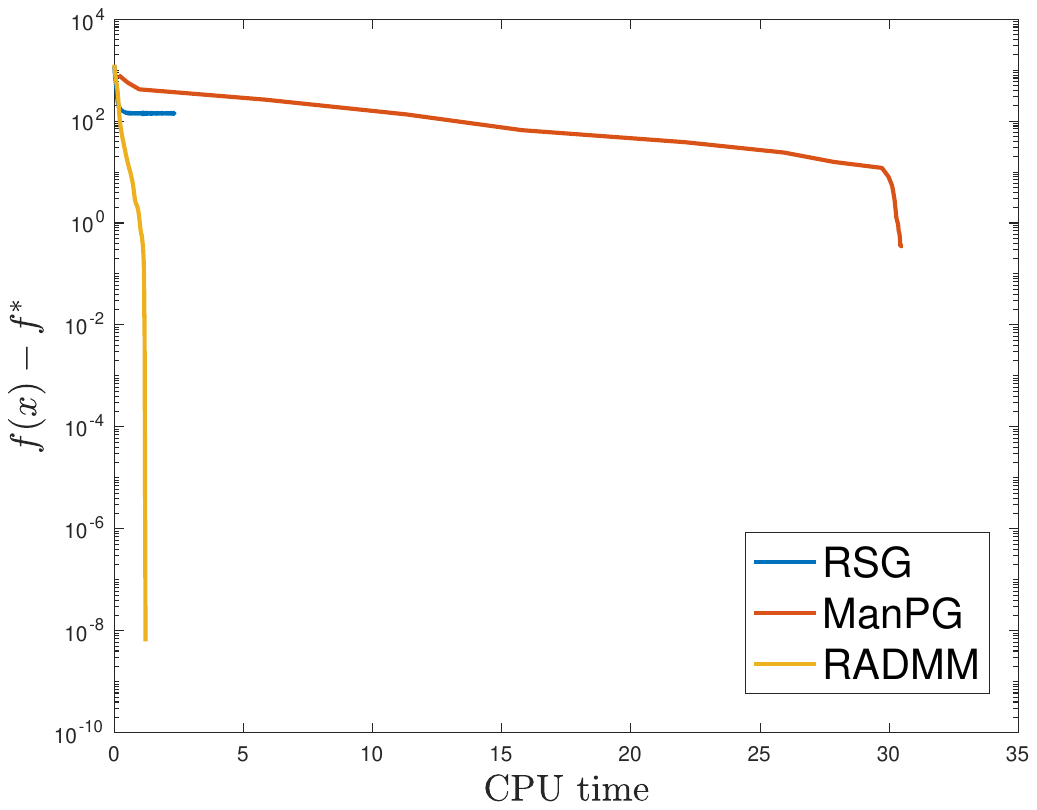}}

\subfigure[$(n, p)=(500,50)$]{\includegraphics[clip, trim=1.5cm 6cm 2cm 6cm,width=0.45\columnwidth]{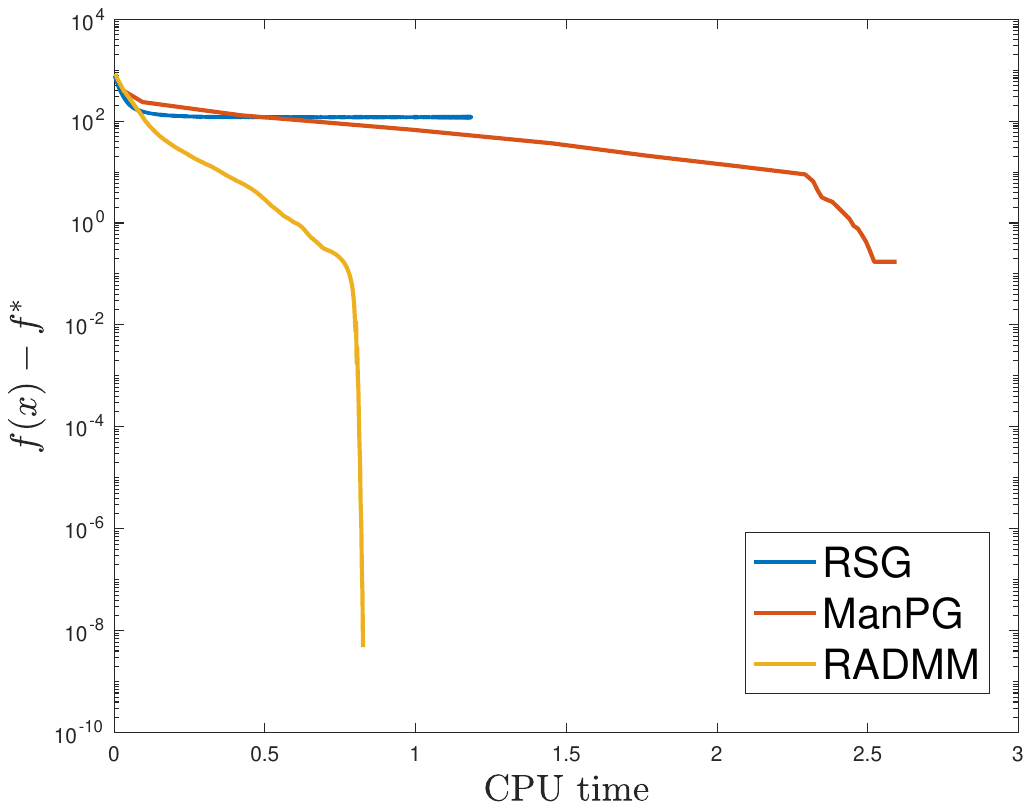}}
\subfigure[$(n, p)=(500,100)$]{\includegraphics[clip, trim=1.5cm 6cm 2cm 6cm,width=0.45\columnwidth]{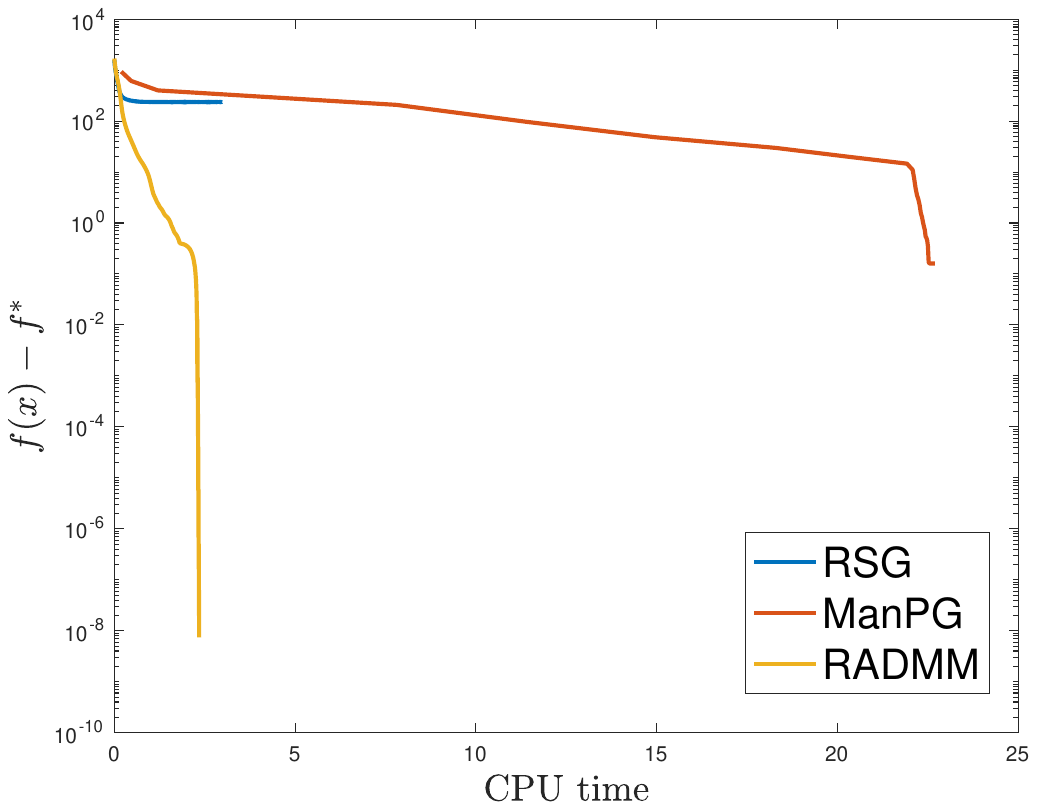}}

\caption{Comparison of the CPU time (in seconds) consumed among the ManPG, RADMM and Riemannian gradient methods for solving \eqref{sPCA} with $\mu=1$. Each figure is averaged for 10 repeated experiments with random initializations.}
\label{fig:spca_cpu_time}
\end{center}
\end{figure}

We also compare our RADMM \eqref{R_admm_Moreau_spca} with SOC~\cite{lai2014splitting} and MADMM~\cite{kovnatsky-madmm-2016}. Before we present the numerical comparisons, we remind the reader that there is no convergence guarantee for SOC and MADMM. The SOC \eqref{soc} algorithm for solving problem \eqref{sPCA} actually solves the following equivalent problem:
\begin{equation}\label{sPCA_split2}
\begin{split}
        \min_{X, Y} & \ -\frac{1}{2}\mathrm{Tr}(X^\top A^\top A X) + \mu\|X\|_{1} \\
        \st & \ X=Y,\ Y\in\St(n,p).
\end{split}
\end{equation}
The SOC iterates as follows.
\begin{equation}\label{soc_spca}
\begin{split}
        & X^{k+1} := \ \argmin_X -\frac{1}{2}\mathrm{Tr}(X^\top A^\top A X) + \mu\|X\|_{1}+\langle\Lambda^k,X-Y^k\rangle +\frac{\rho}{2}\|X - Y^{k}\|_{F}^2 \\
        & Y^{k+1} := \ \argmin_{Y\in\St(n,p)} \langle\Lambda^k,X^{k+1}-Y\rangle + \frac{\rho}{2} \|X^{k+1}-Y\|_{F}^2 \\
        &\Lambda^{k+1} := \ \Lambda^{k} + \rho( X^{k+1} - Y^{k+1}).
\end{split}
\end{equation}
In our numerical experiment, we chose to solve the $X$-subproblem using the  proximal gradient method. The MADMM \eqref{madmm} solves \eqref{sPCA-split}, and iterates as follows:
\begin{equation}\label{madmm_spca}
\begin{split}
        & X^{k+1} := \ \argmin_{X\in\St(n,p)} -\frac{1}{2}\mathrm{Tr}(X^\top A^\top A X)+\langle\Lambda^k,X-Y^k\rangle+\frac{\rho}{2}\|X - Y^{k}\|_{F}^2 \\
        & Y^{k+1} := \ \argmin_Y \mu\|Y\|_{1} + \langle\Lambda^k,X^{k+1}-Y\rangle +\frac{\rho}{2}\|X^{k+1}-Y\|_{F}^2 \\
        &\Lambda^{k+1} := \ \Lambda^{k} + \rho(X^{k+1} - Y^{k+1}).
\end{split}
\end{equation}
In our numerical experiment, we chose to solve the $X$-subproblem using a Riemannian gradient method.

We test our RADMM with SOC and MADMM with the following parameters: for SOC we set $\rho=50$ and $\eta=10^{-2}$, where $\eta$ is the stepsize for the proximal gradient method for solving the $X$-subproblem; for MADMM we set $\rho=100$ and $\eta=10^{-2}$, where $\eta$ is the stepsize for the Riemannian gradient method for solving the $X$-subproblem; for RADMM we set  $\rho=100$, $\eta=10^{-2}$ and $\gamma=10^{-8}$. The parameters are obtained via simple grid searches, also we randomly initialize three algorithms at the same starting point. For all the three algorithms we record the function value and sparsity for the sequence on the manifold, i.e. $X^k$ for MADMM and RADMM, and $Y^k$ for SOC. For each algorithm, we terminate after 100 iterations. We present the function value change curve in Figures \ref{fig:spca_iter2} and \ref{fig:spca_cpu_time2}. We also report the objective function values of the outputs (denoted as ``obj''), the sparsity (the percentage of zero entries, denoted as ``Spa'') and the constraint violation ($\|X^k-Y^k\|_F$ for all three algorithms, denoted as ``infeas'') in Table \ref{table_spca_sparsity2}.
From Figure \ref{fig:spca_iter2} we can see that SOC is more efficient in terms of the iteration number, but from Figure \ref{fig:spca_cpu_time2} we see that RADMM is more efficient in terms of the CPU time. This is exactly because all steps in our RADMM are very easy to compute, and so the per-iteration complexity is very cheap. 

\begin{figure}[t!]
\begin{center}

\setcounter{subfigure}{0}
\subfigure[$(n,p)=(300,50)$]{\includegraphics[clip, trim=1.5cm 6cm 1.5cm 6cm,width=0.44\columnwidth]{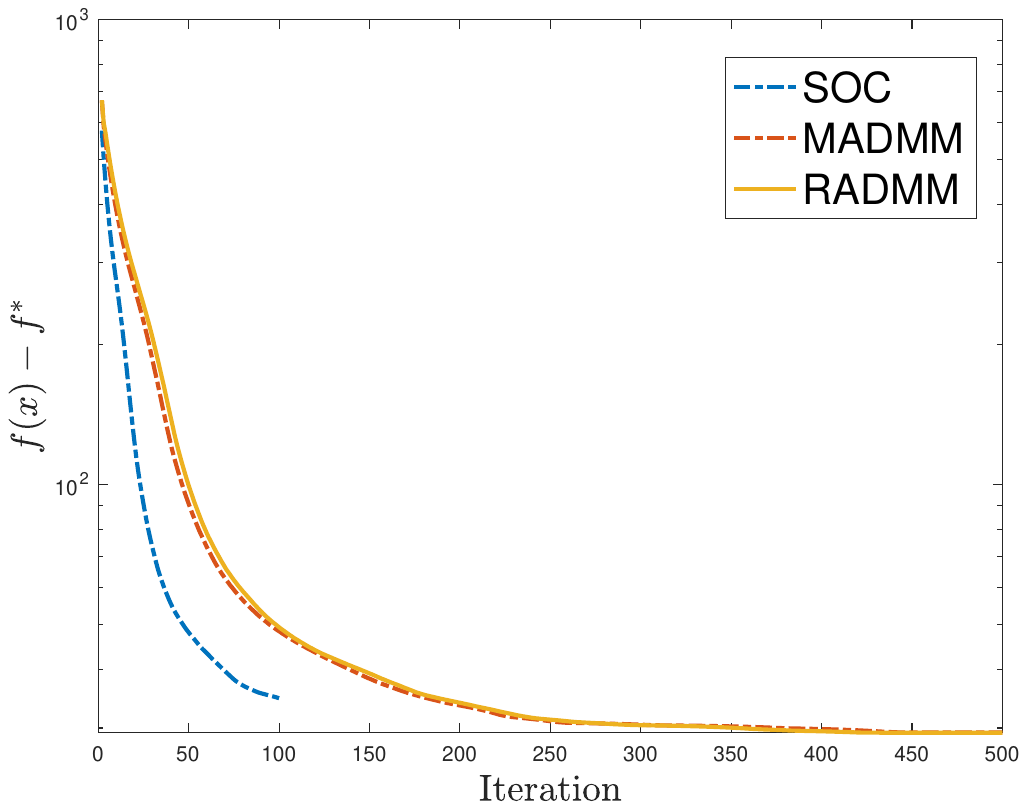}}
\subfigure[$(n,p)=(300,100)$]{\includegraphics[clip, trim=1.5cm 6cm 1.5cm 6cm,width=0.44\columnwidth]{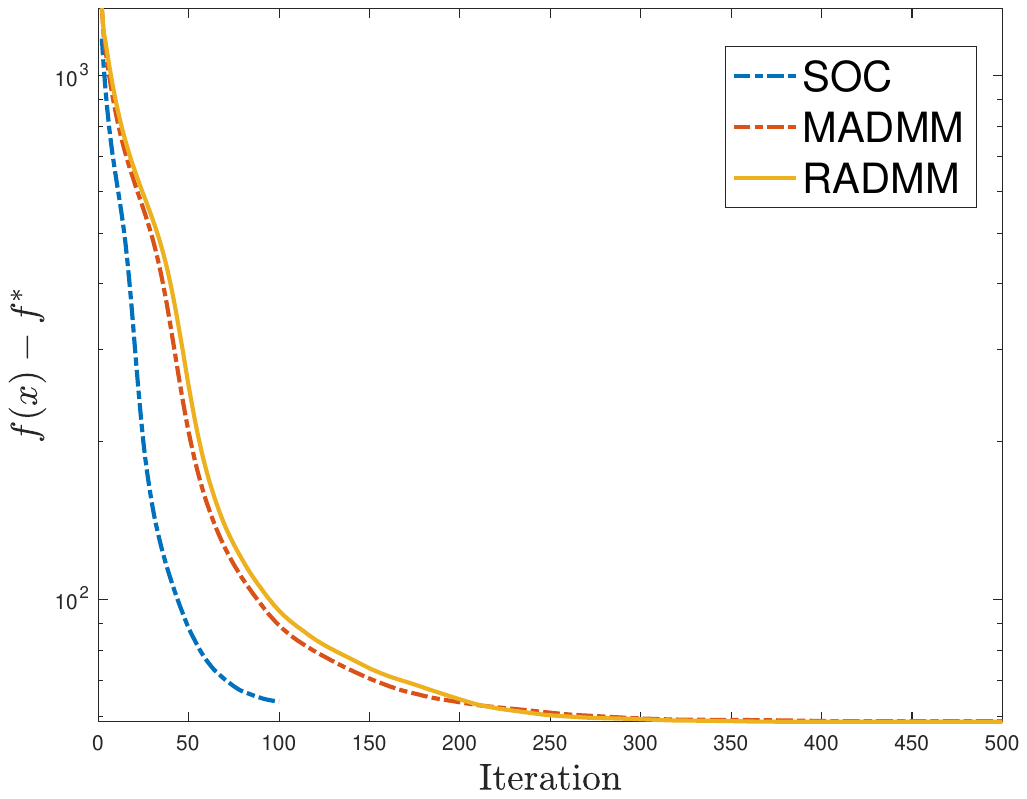}}
\subfigure[$(n,p)=(500,50)$]{\includegraphics[clip, trim=1.5cm 6cm 1.5cm 6cm,width=0.44\columnwidth]{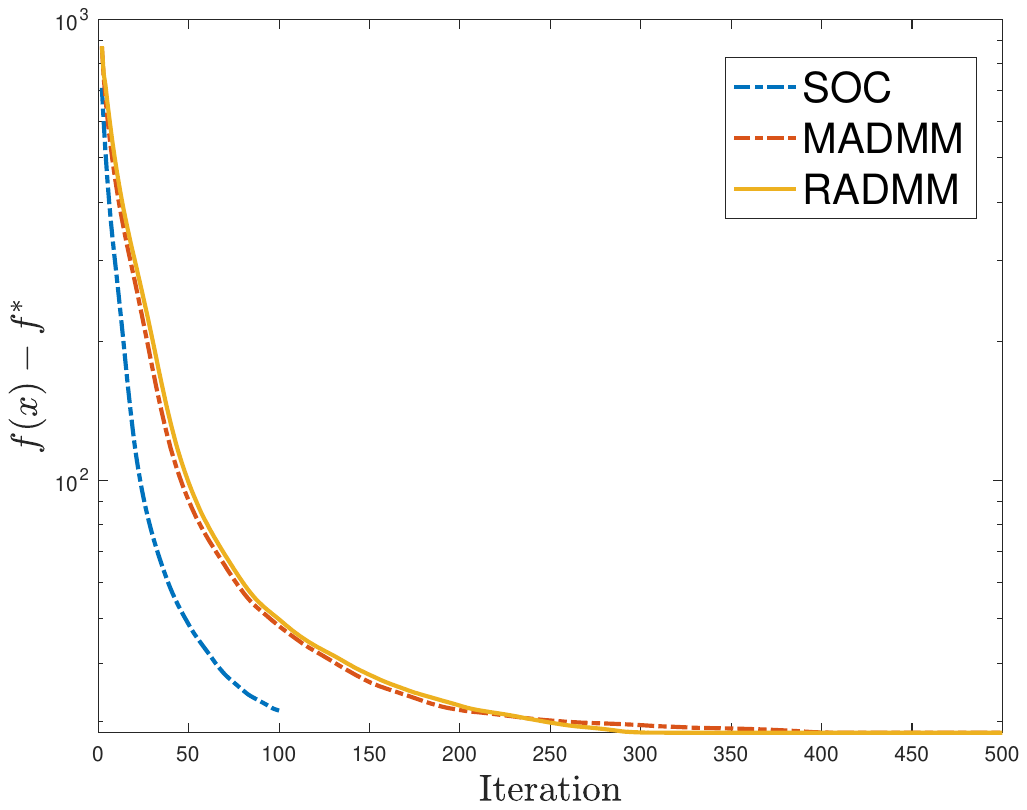}}
\subfigure[$(n, p)=(500,100)$]{\includegraphics[clip, trim=1.5cm 6cm 1.5cm 6cm,width=0.44\columnwidth]{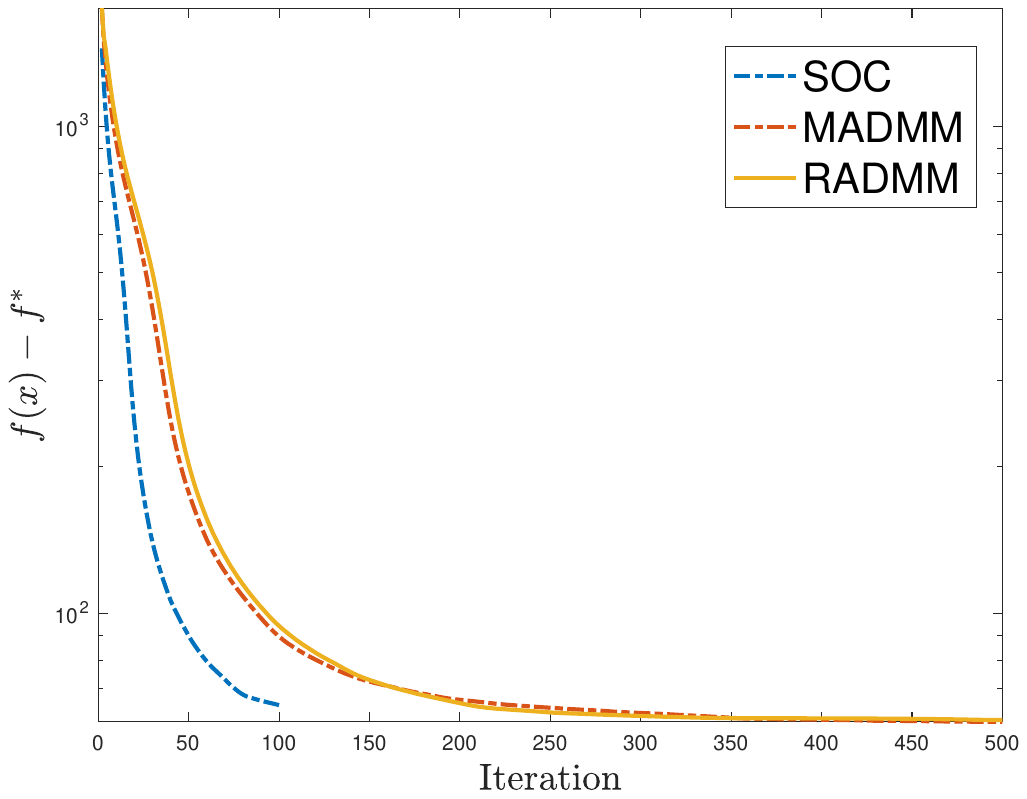}}

\caption{Comparison of SOC, MADMM and RADMM for solving \eqref{sPCA} with $\mu=1$ and with respect to iteration numbers. Each figure is averaged for 10 repeated experiments with random initializations.}
\label{fig:spca_iter2}
\end{center}
\end{figure}

\begin{figure}[t!]
\begin{center}
\setcounter{subfigure}{0}
\subfigure[$n=300, p=50$]{\includegraphics[clip, trim=1.5cm 6cm 1.5cm 6cm,width=0.44\columnwidth]{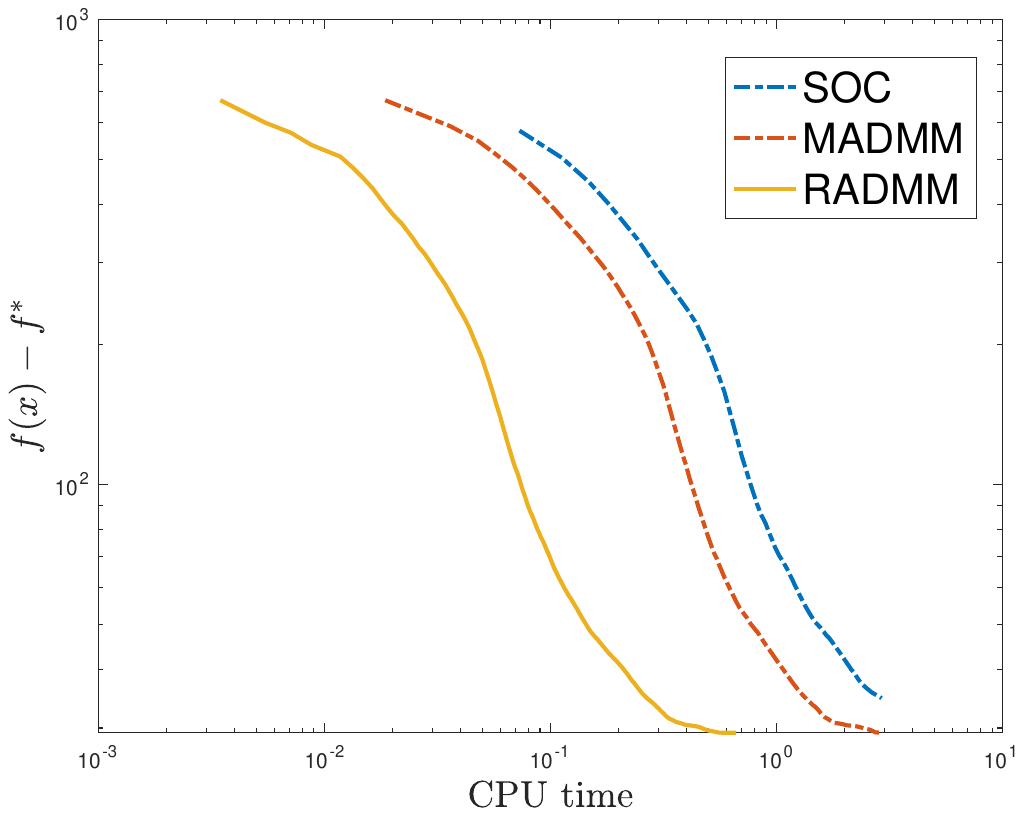}}
\subfigure[$n=300, p=100$]{\includegraphics[clip, trim=1.5cm 6cm 1.5cm 6cm,width=0.44\columnwidth]{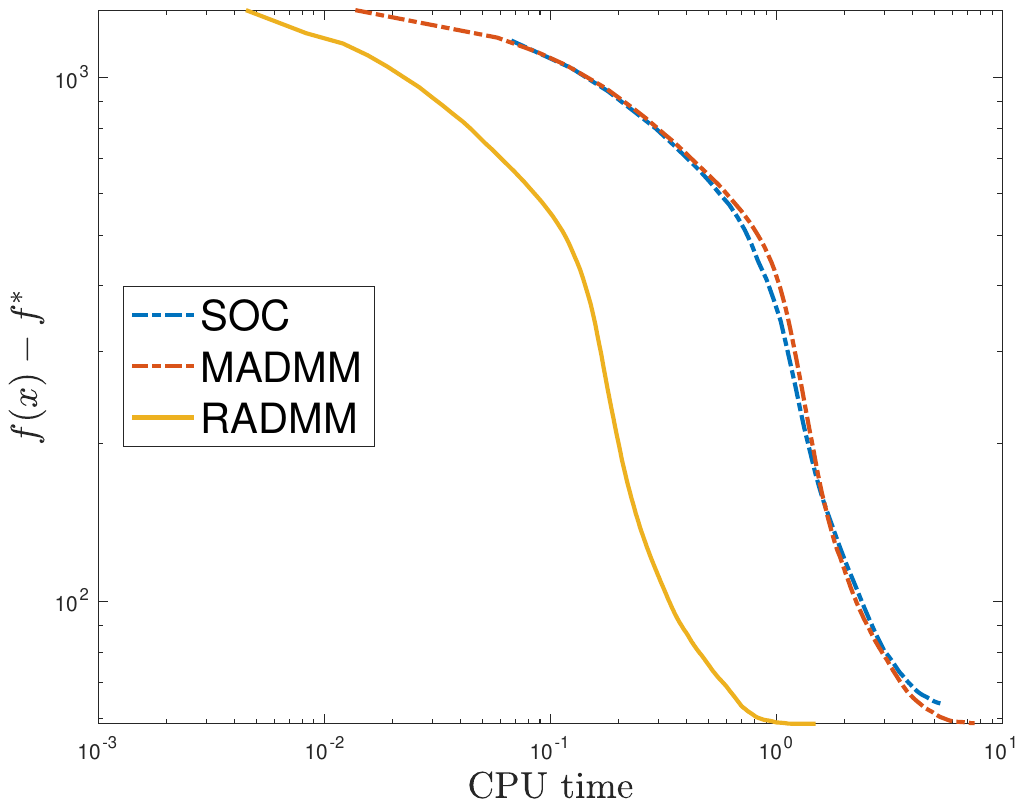}}
\subfigure[$n=500, p=50$]{\includegraphics[clip, trim=1.5cm 6cm 1.5cm 6cm,width=0.44\columnwidth]{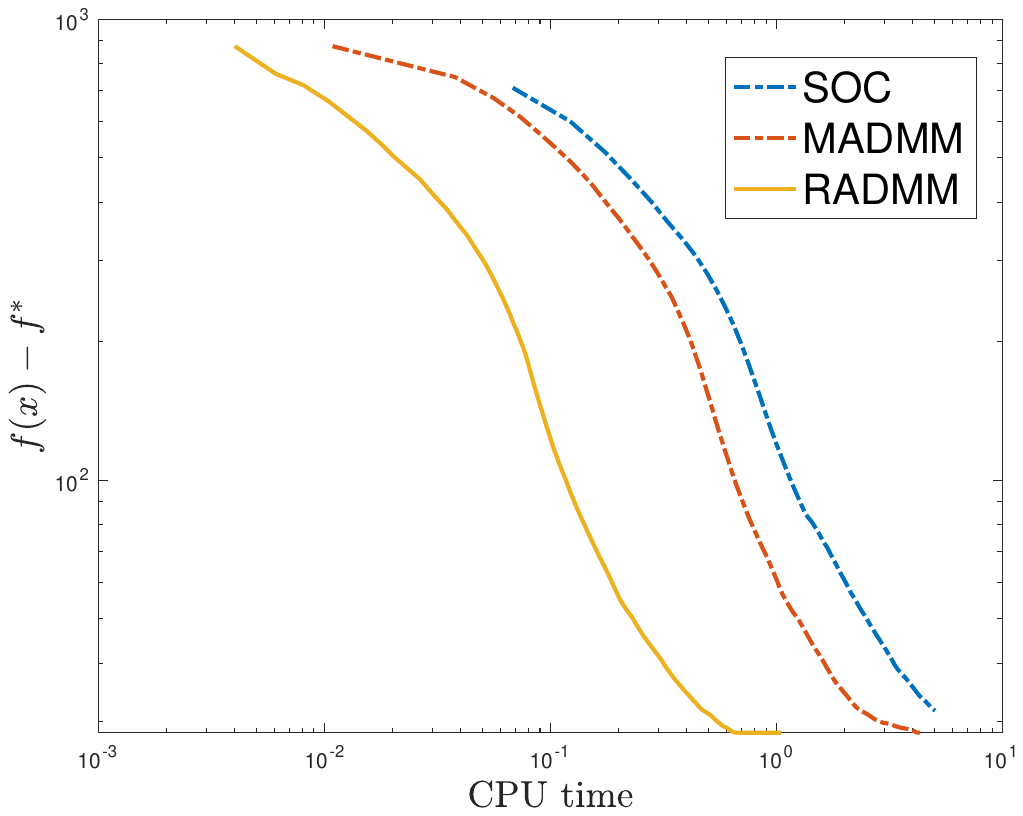}}
\subfigure[$n=500, p=100$]{\includegraphics[clip, trim=1.5cm 6cm 1.5cm 6cm,width=0.44\columnwidth]{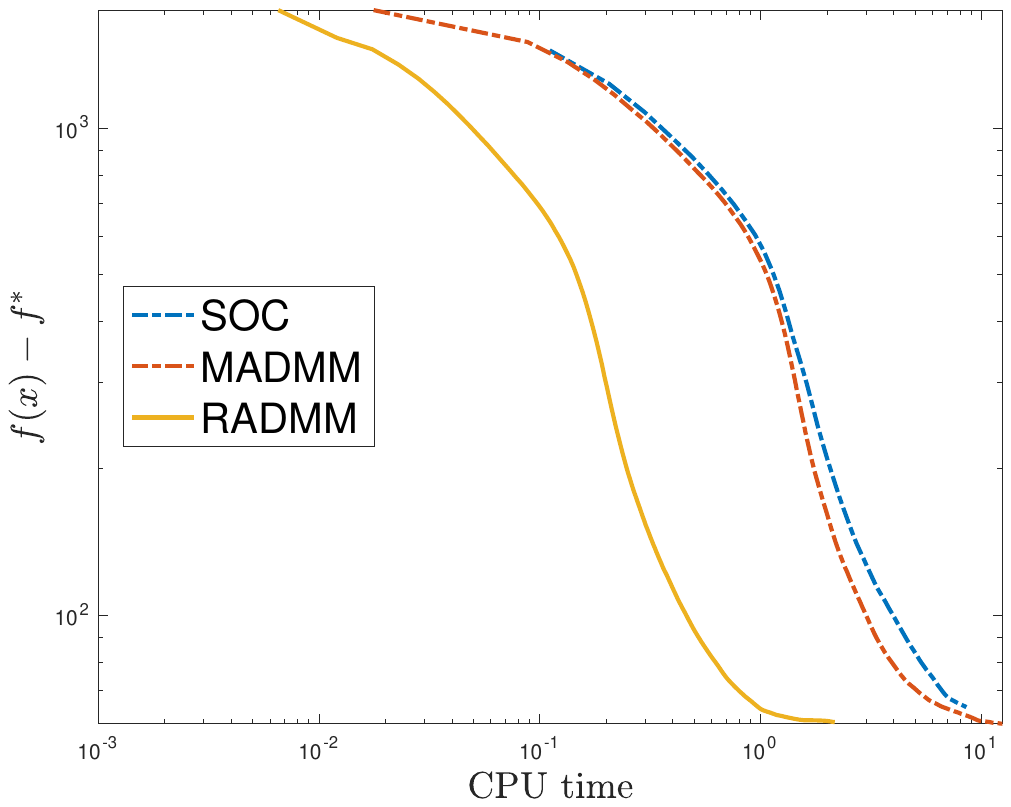}}

\caption{Comparison of SOC, MADMM and RADMM for solving \eqref{sPCA} with $\mu=1$ and with respect to CPU time consumed. Each figure is averaged for 10 repeated experiments with random initializations.}
\label{fig:spca_cpu_time2}
\end{center}
\end{figure}

\begin{table}[ht]
%\vskip 0.15in
\begin{center}
\begin{small}
\begin{tabular}{ c | c  c  c | c  c  c | c  c  c }
 \hline
 Settings & \multicolumn{3}{c|}{SOC} & \multicolumn{3}{c|}{MADMM} & \multicolumn{3}{c}{RADMM} \\
 \hline
 $(n, p)$ & obj & Spa & infeas & obj & Spa & infeas & obj & Spa & infeas \\
 \hline
 $(300, 50)$ & 34.8851 & 0.7609 & 0.0060 & 29.2059 & 0.9967 & 0.0000 & 29.1197 & 0.9967 & 0.0000 \\
 $(300, 100)$ & 66.6870 & 0.6018 & 0.0072 & 59.6483 & 0.9967 & 0.0000 & 59.8210 & 0.9967 & 0.0000 \\
 $(500, 50)$ & 32.7199 & 0.8819 & 0.0040 & 29.4007 & 0.9980 & 0.0000 & 29.5003 & 0.9742 & 0.0000 \\
 $(500, 100)$ & 67.2337 & 0.7558 & 0.0082 & 59.7878 & 0.9977 & 0.0000 & 59.4491 & 0.9980 & 0.0000 \\
 \hline
\end{tabular}
\end{small}
\end{center}
\caption{Comparison of SOC, MADMM and RADMM for solving \eqref{sPCA} with $\mu=1$. The results are averaged for 10 repeated experiments with random initializations.}
\label{table_spca_sparsity2}
%\vskip -0.1in
\end{table}

\subsection{Numerical Experiments on ODL and DPCP}
In this section, we test Algorithm \ref{R_admm_Moreau_grad_step} on the DPCP problem \eqref{DPCP-matrix}, which can be equivalently written as:
\begin{equation}\label{dpcp_split}
\begin{split}
        \min_{X, W} & \ \|W\|_1 \\
        \st, & \ W = Y^\top X,\ X\in\St(n,p).
\end{split}
\end{equation}
Simple calculation shows that Algorithm \ref{R_admm_Moreau_grad_step} for the DPCP problem \eqref{DPCP-matrix} iterates as follows.
\begin{equation}\label{R_admm_Moreau_dpcp}
\begin{split}
    &X^{k+1} := \retr_{X^{k}}(-\eta_{k}\proj_{T_{X^{k}}\St(n,p)}(Y\Lambda^{k}+\rho Y( Y^\top X^{k}-Z^{k}))) \\
    &W^{k+1}:= \mathrm{prox}_{\frac{1+\rho\gamma }{\rho}\|\cdot\|_1}(Y^\top X^{k+1}+\frac{1}{\rho}\Lambda^k)\\
    &Z^{k+1}:= \frac{1}{1/\gamma +\rho}\left(\frac{1}{\gamma }W^{k+1}+\Lambda^k+\rho Y^\top X^{k+1}\right) \\
    &\Lambda^{k+1} := \Lambda^{k} + \rho (Y^\top X^{k+1}-Z^{k+1}).
\end{split}
\end{equation}

We compare the RADMM with iteratively reweighted least squares (IRLS)~\cite{Lerman-RSR-IRLS-2018,Tsakiris-Vidal-2018}, projected subgradient method (PSGM)~\cite{zhu2018dual} and manifold proximal point algorithm (ManPPA)~\cite{chen2019manifold}. Note that the objective of the problem:
\begin{equation}\label{dpcp_again}
\begin{split}
        &\min_{X}\ F(X):=\|Y^\top X\|_1 \\
        &\text{s.t. }\ X\in\St(n,p)=\{X\in\R^{n\times p}|X^\top X = I_{p}\}.
\end{split}
\end{equation}
is separable column-wisely:
\begin{equation}\label{dpcp_columnwise}
\begin{split}
        &\min_{x_1,...,x_p}\ \sum_{i=1}^{p}\|Y^\top x_i\|_1 \\
        &\text{s.t. $\{x_1,...,x_p\}$ is orthonormal set}.
\end{split}
\end{equation}
PSGM and ManPPA conduct the minimization column-wisely. Therefore, in our experiment, we can only record the function value at the outputs of PSGM and ManPPA. Meanwhile the IRLS algorithm that we implemented here is a variant of the original column-wise algorithm for solving \eqref{DPCP-matrix} \cite{Lerman-RSR-IRLS-2018,Tsakiris-Vidal-2018,tsakiris2015dual}. IRLS iterates as follows: first we find the initialization by  $X^0:=\argmin_{X\in\St(n, p)}\|Y^\top X\|_F^2$ and then the iterate is updated by
\begin{equation}\label{IRLS}
    X^{k+1}\leftarrow \argmin_{X\in\St(n, p)}\sum_{i}\|X^\top Y_i\|_2^2 / \max\{\delta, \|(X^{k})^\top Y_i\|_2\}.
\end{equation}
%which enable use to record the function value change since its update takes place on Stiefel manifold.

We follow the same experiment setting as \cite{chen2019manifold}. More specifically, we construct the data to be $Y= [SR, O]$, $S\in\R^{n\times d}$ with orthogonal column vectors, $R\in\R^{d\times p_1}$, $O\in\R^{N\times p_2}$ both with random Gaussian entries. Here $p_1$ and $p_2$ are the numbers of inliers and outliers respectively as described in \cite{chen2020proximal}. In our experiment we set $p = 5$, $p_1=500$ and $p_2=1167$, with different choice of $n$. For our RADMM algorithm we set $\rho=40$, $\gamma=4\cdot 10^{-9}$, $\eta=2\cdot 10^{-4}$. For other algorithms, we use their default parameter settings from the papers \cite{chen2019manifold,zhu2018dual,Tsakiris-Vidal-2018}. For all the algorithm, we terminate them if the difference between two consecutive function values is smaller than $10^{-6}$, %\footnote{For column-wise algorithm PSGM and ManPPA, we terminate each column update with the same criterion.},
i.e.
\[
|F(X^{k+1})-F(X^{k})|<10^{-6}.
\]
We initialize IRLS and RADMM with the same initial point as in \cite{zhu2018dual}. Note that PSGM and ManPPA sequentially solves the column-wise problems, and therefore they do not need the initial point to be on the Stiefel manifold. In Figure \ref{fig:dpcp_theta}, we show how the objective function value changes along with the CPU time. We also record the CPU time and final objective function value in Table \ref{table_dpcp}. For RADMM, we also include the constraint violation (i.e. $\|W^k - Y^\top X^k\|_F$, denoted as ``infeas'' in the table) in Table \ref{table_dpcp}. It can be seen from Figure \ref{fig:dpcp_theta} and Table \ref{table_dpcp} that RADMM outputs the other three algorithms in terms of the objective function value.

\begin{figure}[t!]
\begin{center}
\subfigure[$(n, p)=(30, 5)$]{\includegraphics[clip, trim=1.5cm 6cm 1.5cm 7cm, width=0.44\columnwidth]{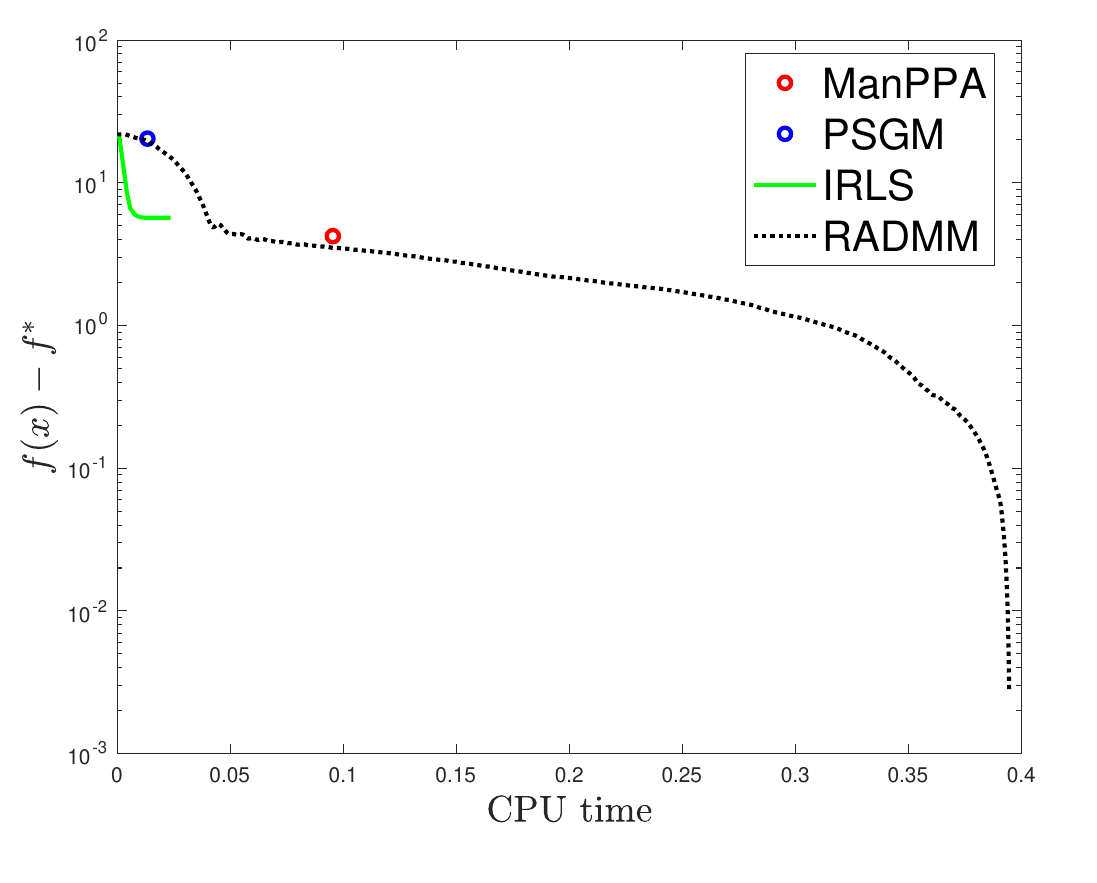}}
\subfigure[$(n, p)=(50, 5)$]{\includegraphics[clip, trim=1.5cm 6cm 1.5cm 7cm, width=0.44\columnwidth]{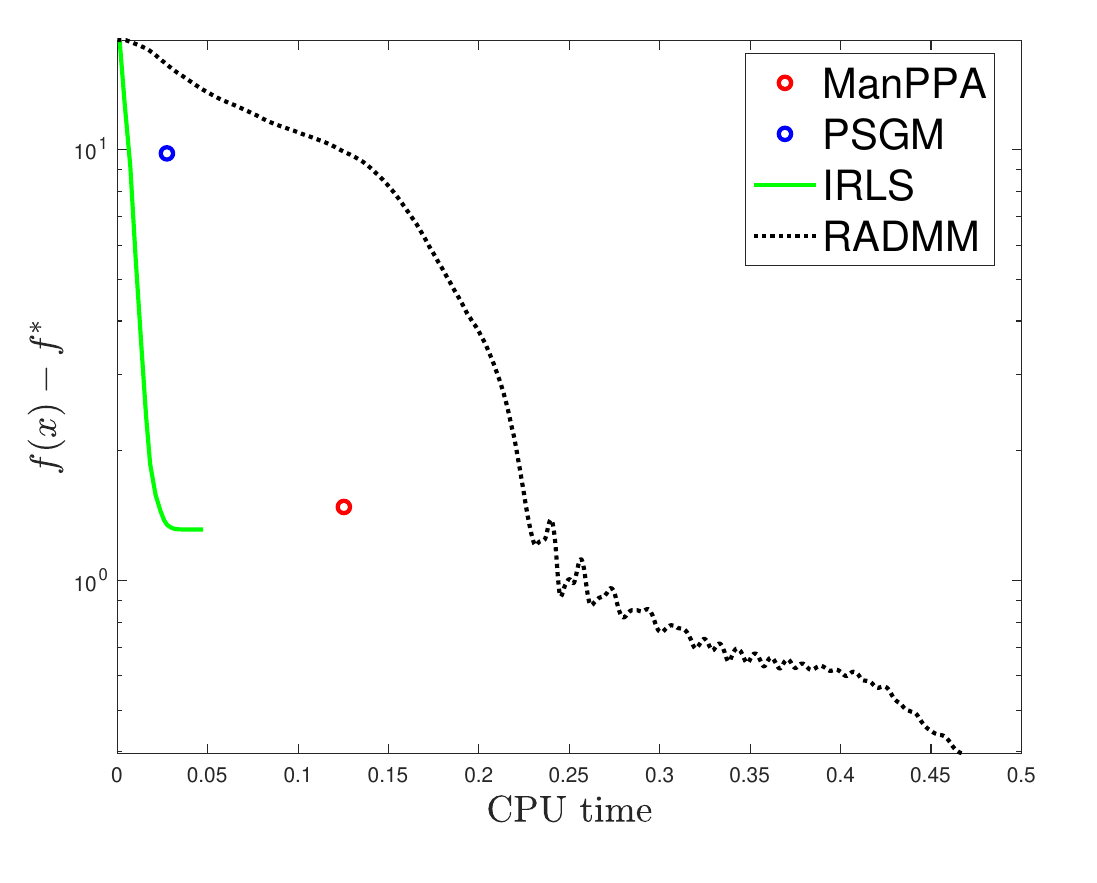}}

\caption{Function value $\|Y^\top X^{k}\|_1$ versus CPU time. In this experiment we set $n\in\{30, 50\}$, $p=5$, and we set the number of Inlier $p_1=500$ and Outlier $p_2=1167$. The experiments are repeated and averaged for 10 times with random initialization.}
\label{fig:dpcp_theta}
\end{center}
\end{figure}

\begin{table}[ht]
%\vskip 0.15in
\begin{center}
\begin{small}
\begin{tabular}{c | c  c | c  c | c  c | c  c  c}
 \hline
 Settings & \multicolumn{2}{c|}{PSGM} & \multicolumn{2}{c|}{IRLS} & \multicolumn{2}{c|}{ManPPA} & \multicolumn{3}{c}{RADMM} \\
 \hline
 $(n, p)$ & obj & CPU & obj & CPU & obj & CPU & obj & CPU & vio \\
 \hline
  $(30, 5)$ & 180.59 & 0.0131 & 177.66 & 0.0230 & 177.90 & 0.1164 & 173.28 & 0.3177 & 0.0003 \\
  $(50, 5)$ & 141.66 & 0.0215 & 142.61 & 0.0404 & 138.78 & 0.1820 & 136.62 & 0.3971 & 0.0007 \\
  $(70, 5)$ & 125.94 & 0.0429 & 118.97 & 0.0881 & 119.50 & 0.3532 & 116.39 & 0.4526 & 0.0074 \\
 \hline
\end{tabular}
\end{small}
\end{center}
\caption{Summary of function value, CPU time (seconds) of proposed RADMM Algorithm \eqref{R_admm_Moreau_dpcp}, comparing with PSGM~\cite{zhu2018dual}, IRLS~\cite{Lerman-RSR-IRLS-2018,tsakiris2015dual} and ManPPA~\cite{chen2019manifold} algorithm. The results are averaged for 10 repeated experiments with random generated data. In this experiment we set $p_1=500$ and $p_2=1167$.}
\label{table_dpcp}
%\vskip -0.1in
\end{table}

We also compare our RADMM \eqref{R_admm_Moreau_dpcp} with SOC~\cite{lai2014splitting} and MADMM~\cite{kovnatsky-madmm-2016}. The SOC \eqref{soc} algorithm for problem \eqref{DPCP-matrix} actually solves the following equivalent problem:
\[
\begin{split}
        \min_{X, W} & \ \|Y^\top X\|_1 \\
        \st, & \ X = W,\ W\in\St(n,p),
\end{split}
\]
and it iterates as:
\begin{equation}\label{soc_dpcp}
\begin{split}
        X^{k+1} := & \ \argmin_X \|Y^\top X\|_{1}+\langle\Lambda^k,X-W^k\rangle +\frac{\rho}{2}\|X - W^{k}\|_{F}^2 \\
        W^{k+1} := & \ \argmin_{W\in\St(n,p)} \langle\Lambda^k,X^{k+1}-W\rangle +\frac{\rho}{2}\|X^{k+1} - W\|_{F}^2 \\
        \Lambda^{k+1} := & \ \Lambda^{k} + \rho(X^{k+1} - W^{k+1}).
\end{split}
\end{equation}
In our experiment, we chose to solve the $X$-subproblem by a subgradient method~\cite{beck2017first}. MADMM \eqref{madmm} solves \eqref{dpcp_split}, and updates the iterates as follows:
\begin{equation}\label{madmm_dpcp}
\begin{split}
        X^{k+1} := & \ \argmin_{X\in\St(n,p)} \langle \Lambda^k,Y^\top X-W^k\rangle + \frac{\rho}{2}\|Y^\top X - W^{k} \|_{F}^2 \\
        W^{k+1} := & \ \argmin_W \|W\|_{1} +\langle \Lambda^k,Y^\top X^{k+1}-W\rangle + \frac{\rho}{2}\|Y^\top X^{k+1}-W \|_{F}^2 \\
        \Lambda^{k+1} := & \ \Lambda^{k} + \rho(Y^\top X^{k+1} - W^{k+1}).
\end{split}
\end{equation}
In our experiment, we chose to solve the $X$-subproblem by a Riemannian gradient descent method.

The parameters are set as follows. For SOC we set $\rho=50$ and $\eta=5\cdot 10^{-6}$, where $\eta$ is the stepsize for the subgradient step; for MADMM we set $\rho=50$ and $\eta=10^{-6}$, where $\eta$ is the stepsize for the $X$ update; for RADMM we set $\rho=50$, $\eta=10^{-4}$ and $\gamma = 10^{-9}$. Again, the parameters are obtained via simple grid searches, also we randomly initialize three algorithms at the same starting point. For all the three algorithms we record the function value for the sequence on the manifold, i.e. $X^k$ for MADMM and RADMM, and $W^k$ for SOC. We terminate the algorithms with termination check $|F(X^{k+1})-F(X^{k})|\leq 10^{-8}$. We record the objective function values in Figures \ref{fig:dpcp_iter2} and \ref{fig:dpcp_cpu_time2}. We also report the objective function values of the final output (denoted as ``obj'') and the constraint violation ($\|X^k-W^k\|_F$ for SOC and $\|Y^\top X^k-W^k\|_F$ for MADMM and RADMM, denoted as ``infeas'') in Table \ref{table_dpcp2}. It can be seen from Figure \ref{fig:dpcp_cpu_time2} and Table \ref{table_dpcp2} that RADMM is more efficient in terms of CPU time, despite small constraint violation.

\begin{figure}[t!]
\begin{center}

\subfigure{\includegraphics[clip, trim=1.5cm 6cm 1.5cm 6cm,width=0.32\columnwidth]{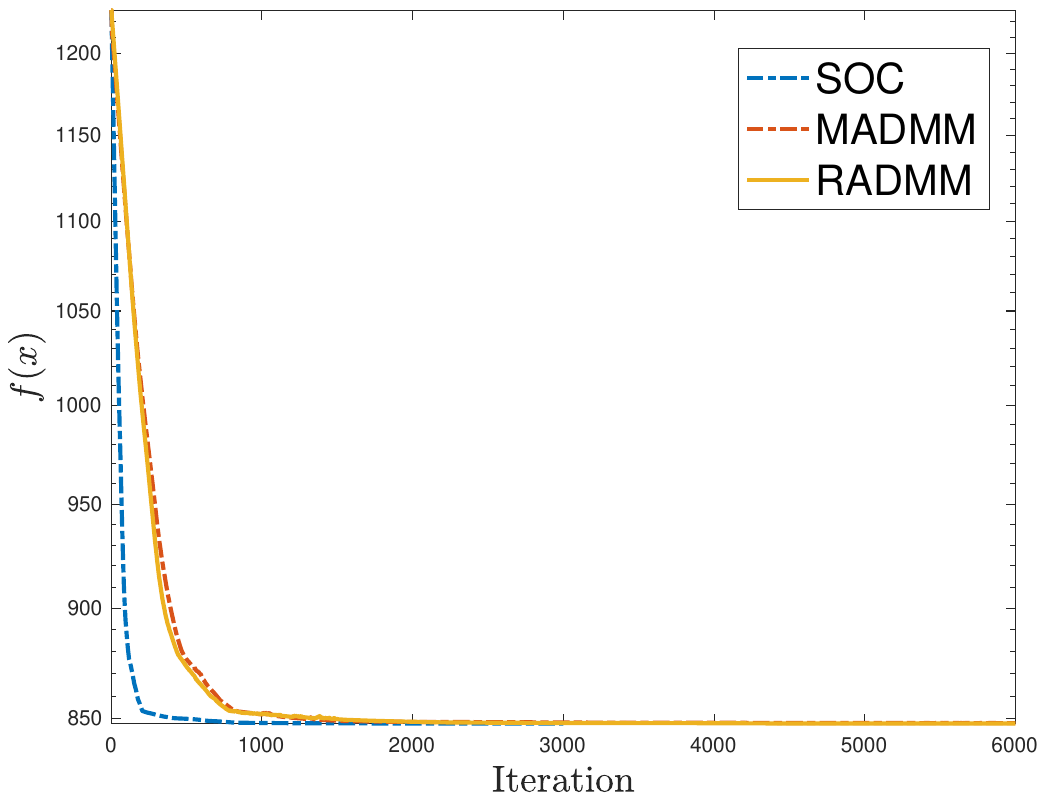}}
\subfigure{\includegraphics[clip, trim=1.5cm 6cm 1.5cm 6cm,width=0.32\columnwidth]{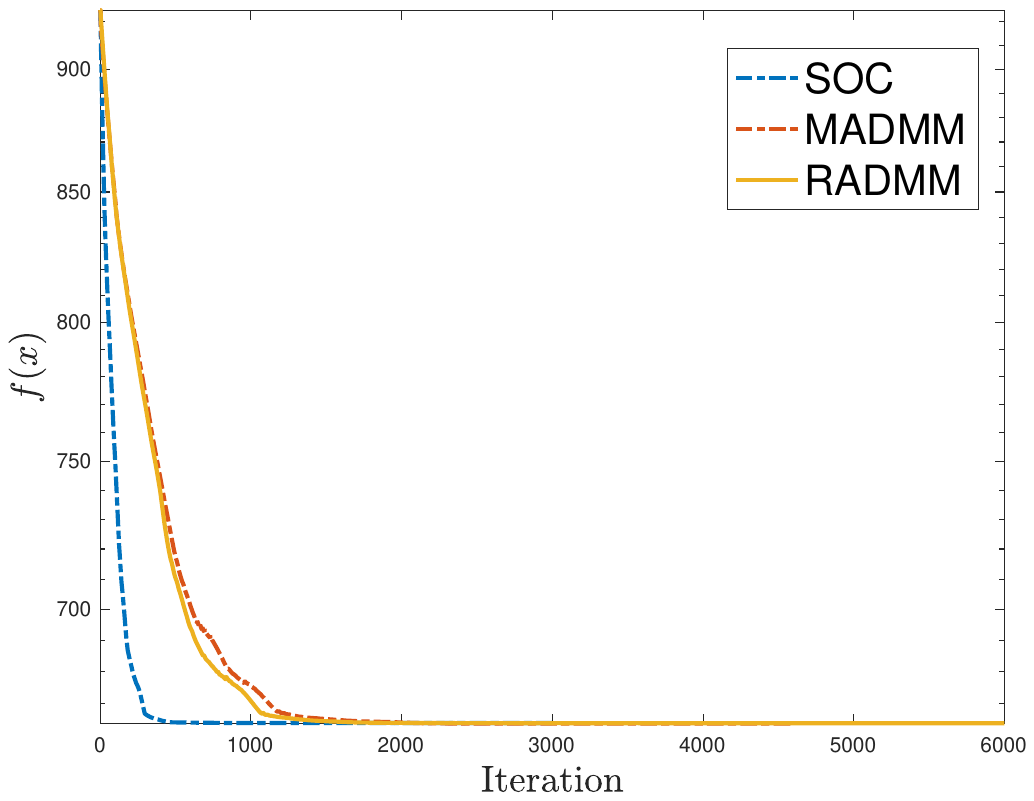}}
\subfigure{\includegraphics[clip, trim=1.5cm 6cm 1.5cm 6cm,width=0.32\columnwidth]{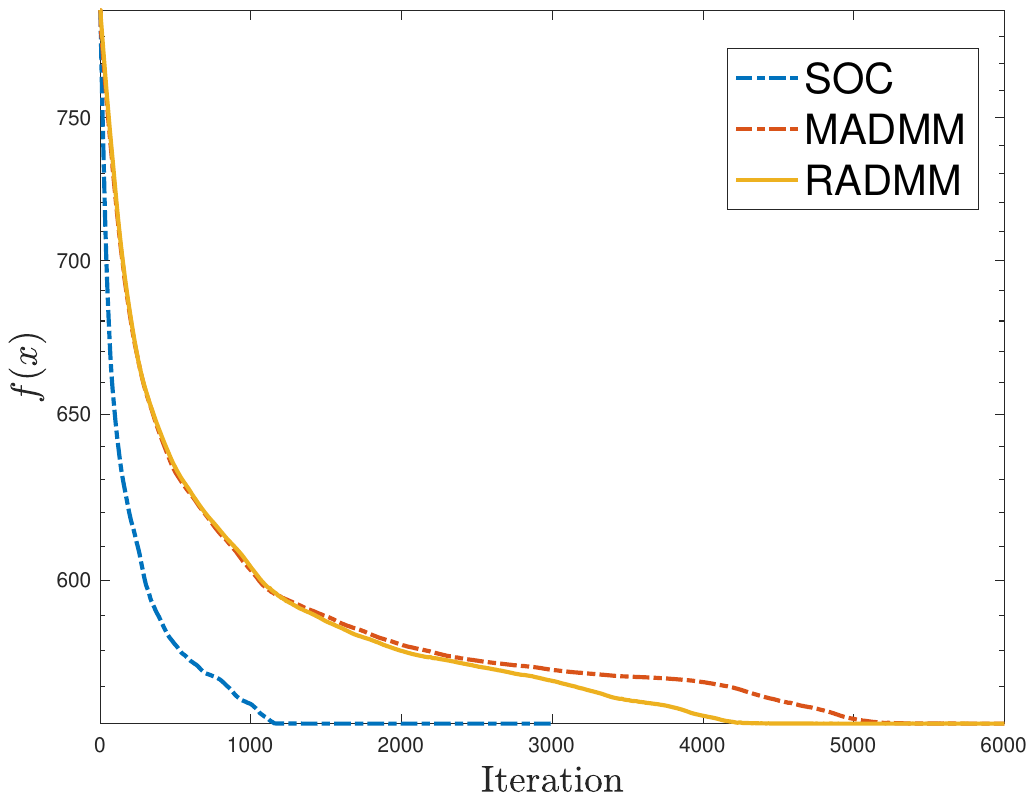}}

\setcounter{subfigure}{0}
\subfigure[$(n, p) = (30, 5)$]{\includegraphics[clip, trim=1.5cm 6cm 1.5cm 6cm,width=0.32\columnwidth]{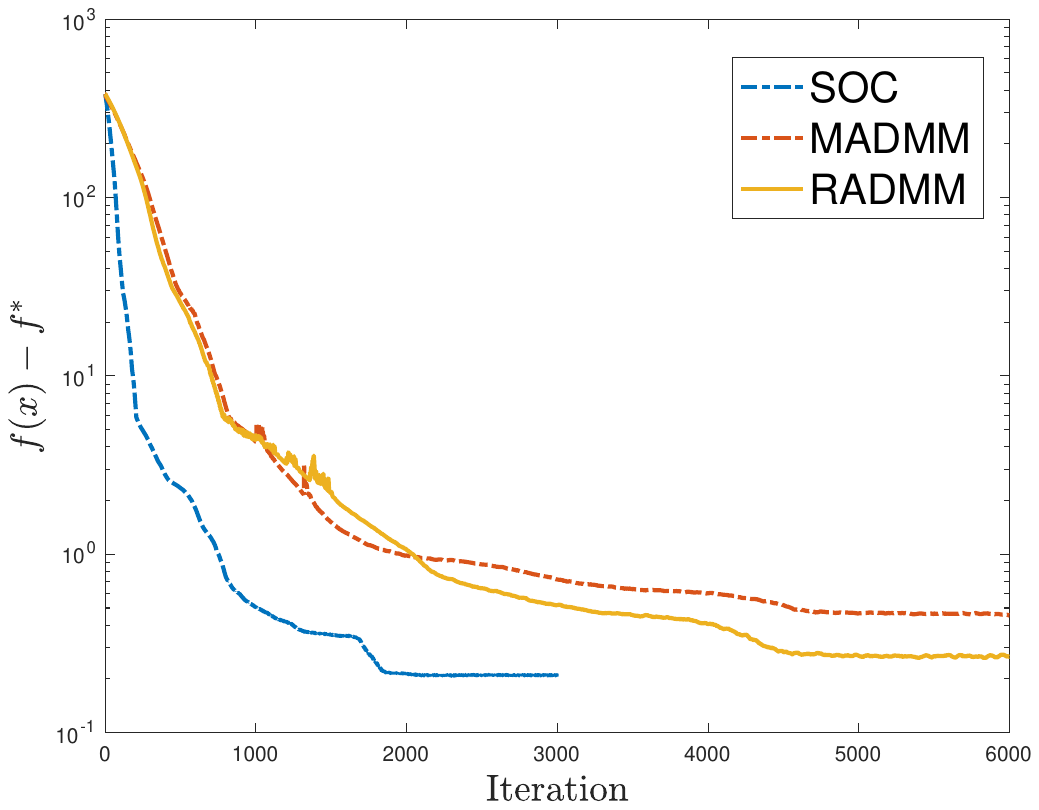}}
\subfigure[$(n, p) = (50, 5)$]{\includegraphics[clip, trim=1.5cm 6cm 1.5cm 6cm,width=0.32\columnwidth]{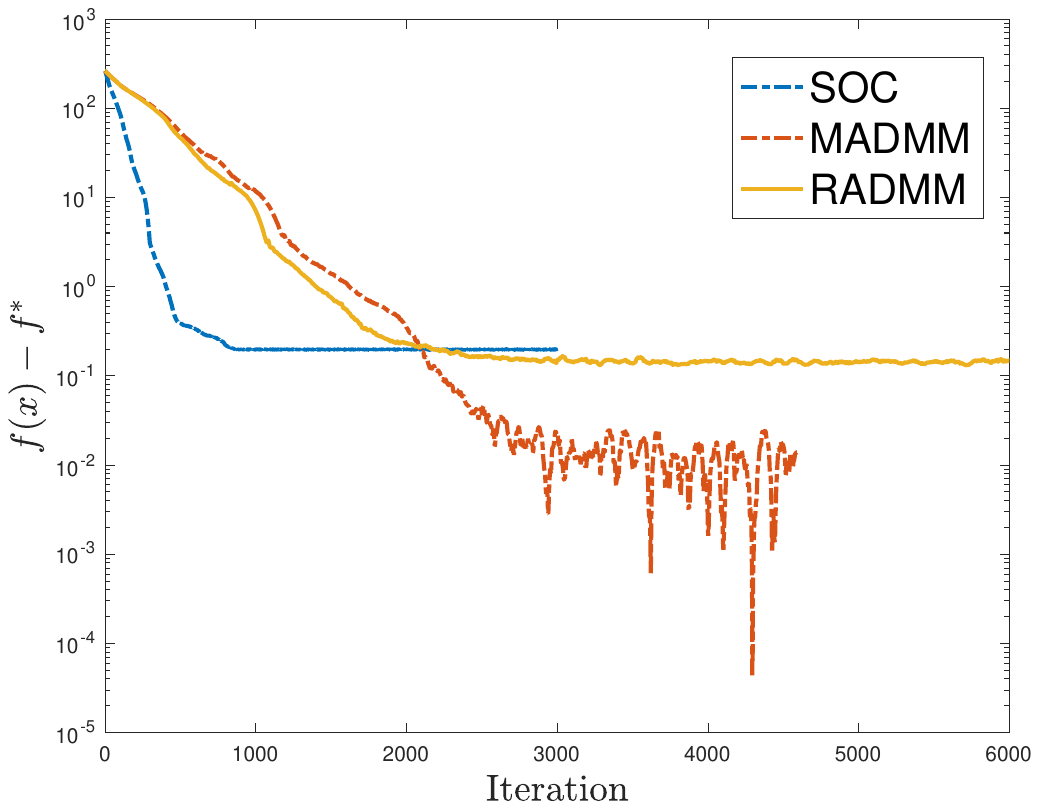}}
\subfigure[$(n, p) = (70, 5)$]{\includegraphics[clip, trim=1.5cm 6cm 1.5cm 6cm,width=0.32\columnwidth]{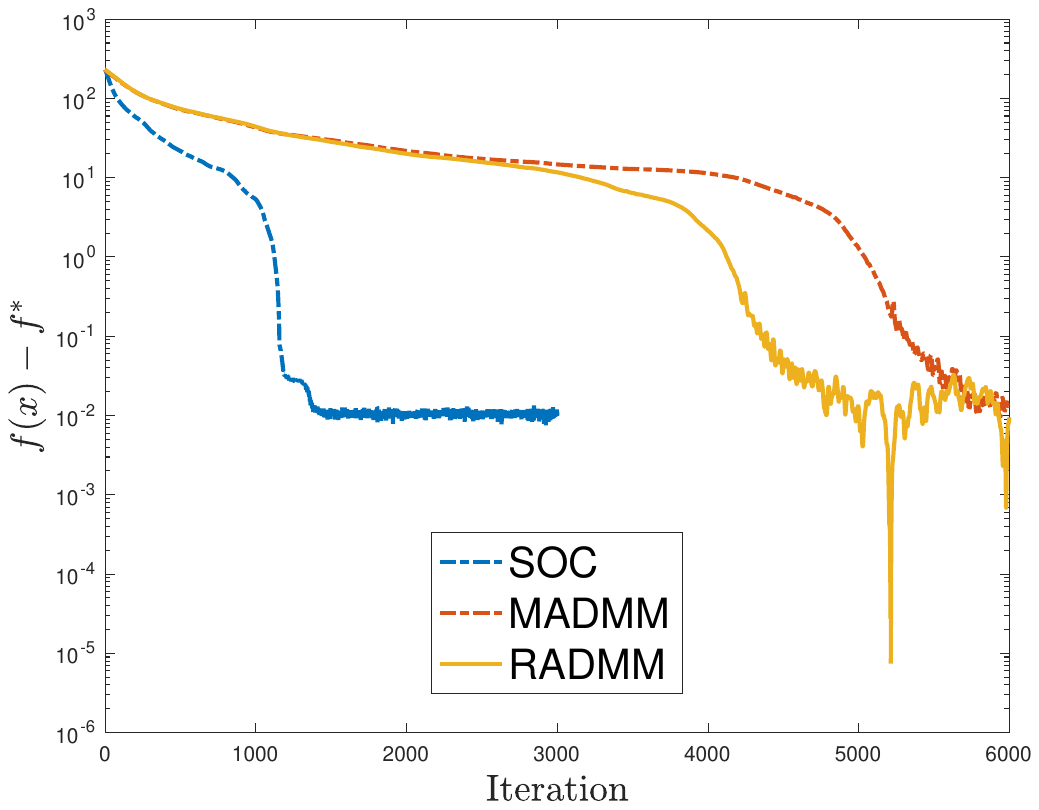}}

\caption{Comparison of SOC, MADMM and RADMM for solving \eqref{DPCP-matrix} with respect to iteration number. The upper row is the plot of function value $f(X^k)$, and the lower row is the plot of $f(X^k)-f^\star$. Note that here $f^*$ is still taken as the minimum function value output by all three algorithms. Each figure is averaged for 5 repeated experiments with random initializations.}
\label{fig:dpcp_iter2}
\end{center}
\end{figure}

\begin{figure}[t!]
\begin{center}
\subfigure{\includegraphics[clip, trim=1.5cm 6cm 1.5cm 6cm,width=0.32\columnwidth]{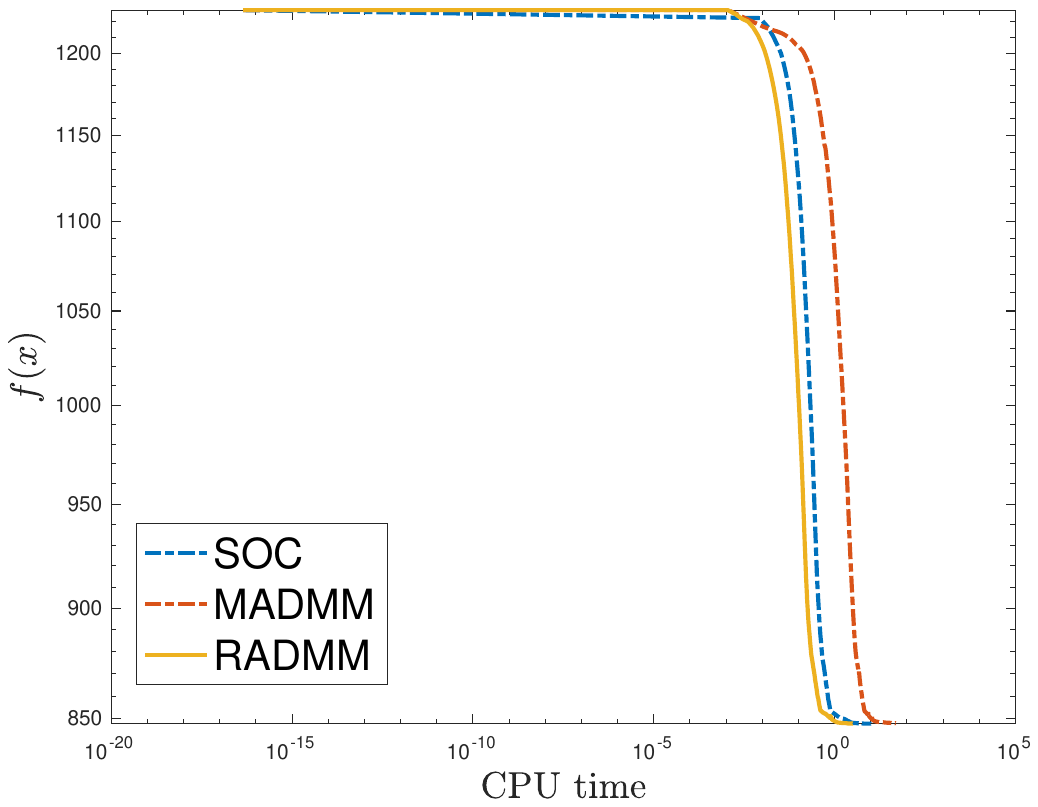}}
\subfigure{\includegraphics[clip, trim=1.5cm 6cm 1.5cm 6cm,width=0.32\columnwidth]{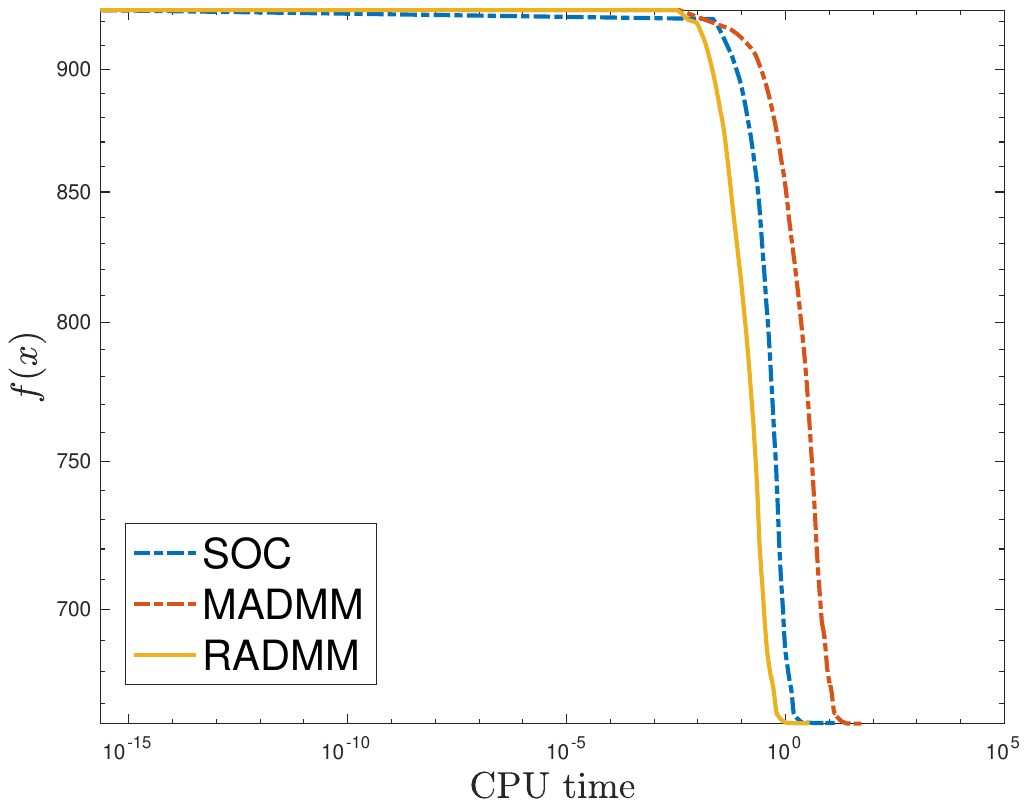}}
\subfigure{\includegraphics[clip, trim=1.5cm 6cm 1.5cm 6cm,width=0.32\columnwidth]{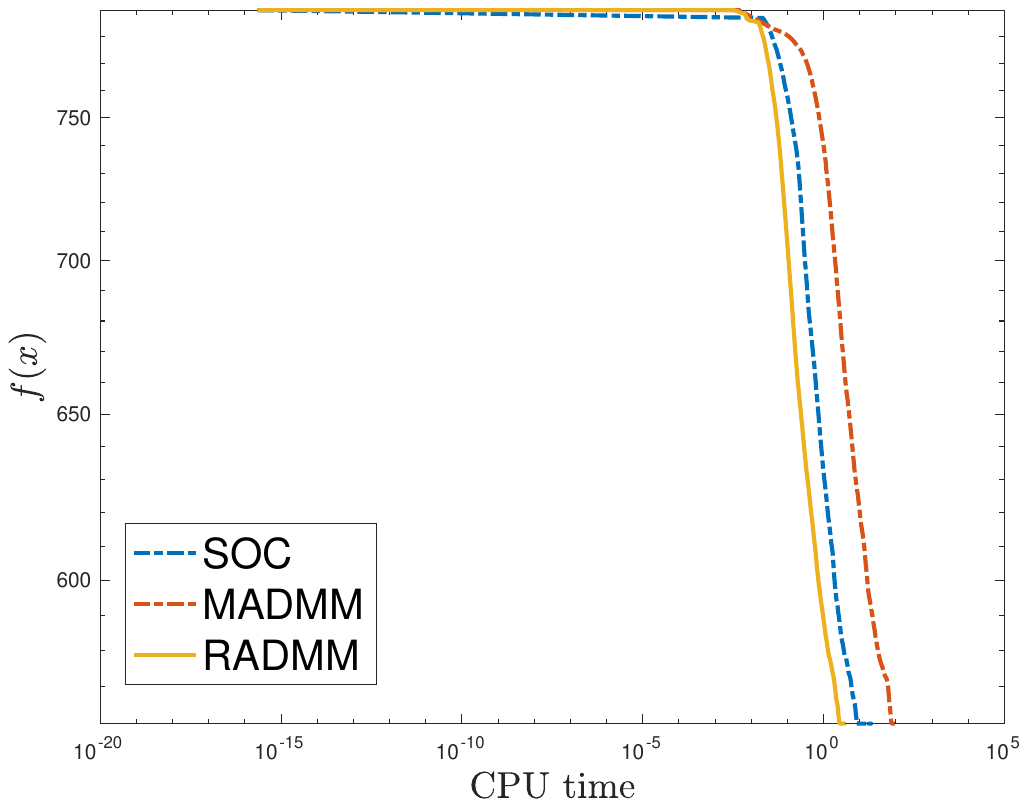}}

\setcounter{subfigure}{0}
\subfigure[$(n, p) = (30, 5)$]{\includegraphics[clip, trim=1.5cm 6cm 1.5cm 6cm,width=0.32\columnwidth]{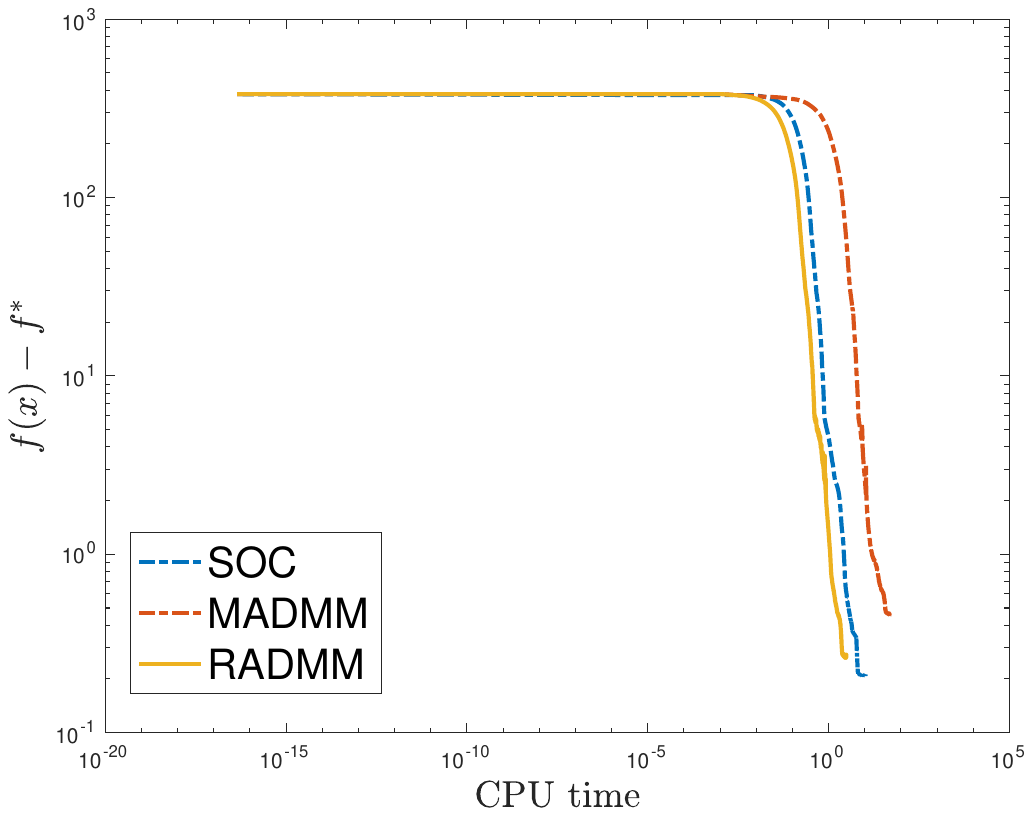}}
\subfigure[$(n, p) = (50, 5)$]{\includegraphics[clip, trim=1.5cm 6cm 1.5cm 6cm,width=0.32\columnwidth]{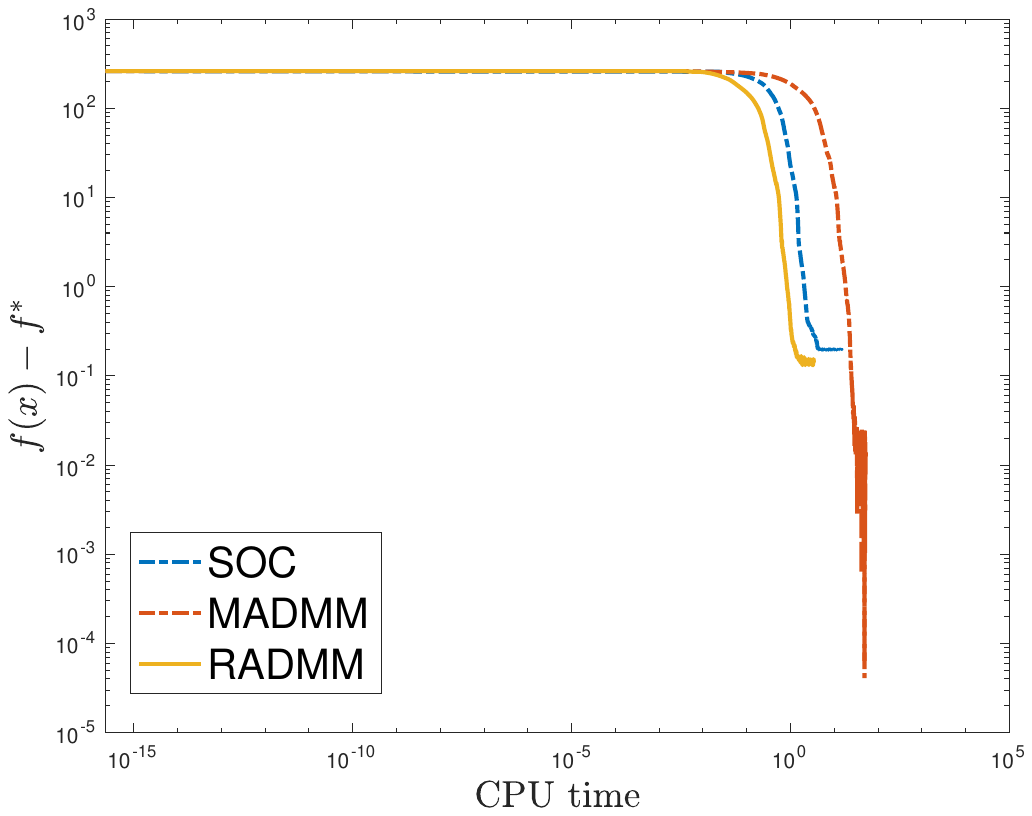}}
\subfigure[$(n, p) = (70, 5)$]{\includegraphics[clip, trim=1.5cm 6cm 1.5cm 6cm,width=0.32\columnwidth]{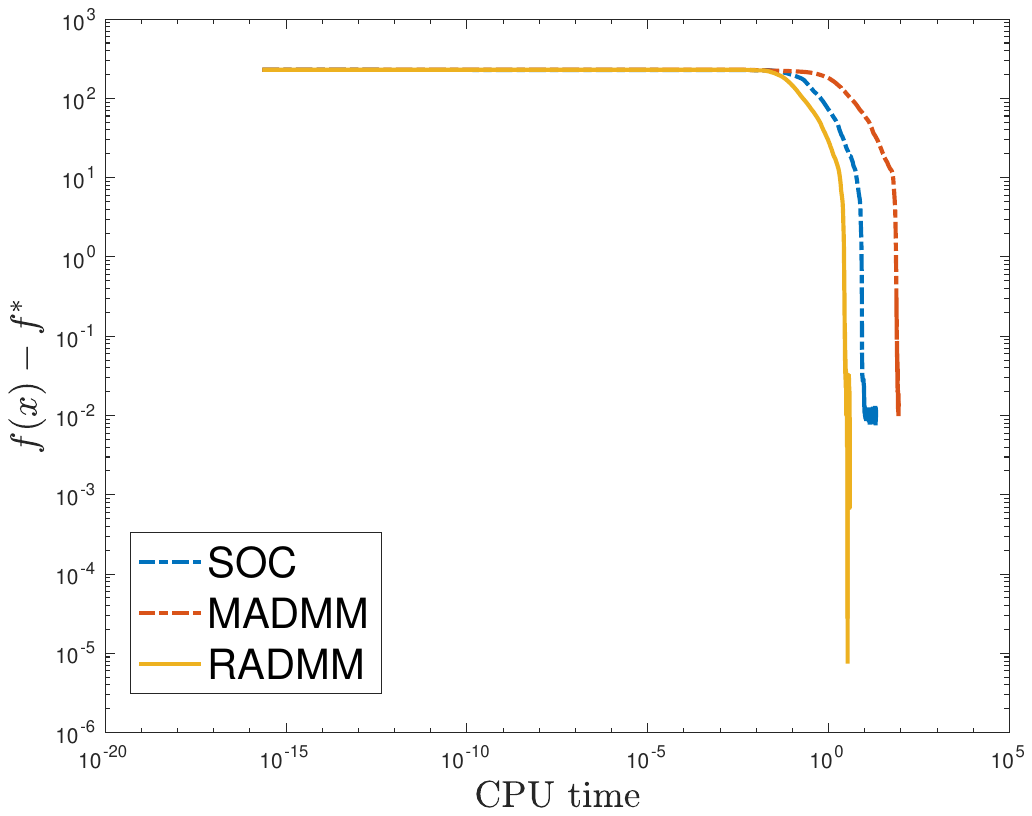}}

\caption{Comparison of SOC, MADMM and RADMM for solving \eqref{DPCP-matrix} with respect to CPU time consumed. The upper row is the plot of function value $f(X^k)$, and the lower row is the plot of $f(X^k)-f^\star$. Note that here $f^*$ is still taken as the minimum function value output by all three algorithms. Each figure is averaged for 5 repeated experiments with random initializations.}
\label{fig:dpcp_cpu_time2}
\end{center}
\end{figure}

\begin{table}[ht]
%\vskip 0.15in
\begin{center}
\begin{small}
\begin{tabular}{ c | c c c | c c c | c c c }
 \hline
 Settings & \multicolumn{3}{c|}{SOC} & \multicolumn{3}{c|}{MADMM} & \multicolumn{3}{c}{RADMM} \\
 \hline
 $(n, p)$ & obj & CPU & infeas & obj & CPU & infeas & obj & CPU & infeas \\
 \hline
 $(30, 5)$ & 847.7757 & 10.5560 & 0.0000 & 848.0198 & 48.7469 & 0.0021 & 847.8329 & 3.1029 & 0.0020 \\
 $(50, 5)$ & 659.8181 & 10.8222 & 0.0001 & 660.1064 & 42.6082 & 0.0026 & 659.6928 & 2.1563 & 0.0376 \\
 $(70, 5)$ & 559.9930 & 20.5054 & 0.0001 & 559.9970 & 86.9119 & 0.0018 & 559.9925 & 3.8470 & 0.0019  \\
 \hline
\end{tabular}
\end{small}
\end{center}
\caption{Comparison of SOC, MADMM and RADMM for solving DPCP problem, with the stopping criterion $\left|F\left(X^{k+1}\right)-F\left(X^k\right)\right|<10^{-8}$. The results are averaged for 5 repeated experiments with random initializations.}
\label{table_dpcp2}
%\vskip -0.1in
\end{table}

\section{Conclusion}\label{sec_conclusion}

In this paper, we proposed a Riemannian ADMM for solving a class of Riemannian optimization problem with nonsmooth objective function. All steps of our Riemannian ADMM are easy to compute and implement, which gives the potential to be applied to solving large-scale problems.
Our method is based on a Moreau envelop smoothing technique. How to design ADMM for solving \eqref{problem_nonsmooth} without smoothing remains an open question for future work.

\clearpage
{
%\small
\bibliography{bibfile,manifold,manifold1,NSF-nonsmooth-manifold,MKSSC_reference}
\bibliographystyle{plain}
}

\end{document}